\pdfoutput=1
\documentclass{amsart}
\usepackage{booktabs}
\usepackage{amssymb}
\usepackage{verbatim}
\makeindex

\usepackage{hyperref}
\hypersetup{
colorlinks=true,       % false: boxed links; true: colored links
    linkcolor=blue,          % color of internal links
    citecolor=blue,        % color of links to bibliography
    filecolor=blue,      % color of file links
    urlcolor=blue           % color of external links
}
\usepackage[all]{xy}
\usepackage{tikz}
\usepackage{todonotes}
\usepackage[alphabetic,backrefs,lite]{amsrefs}
\usepackage{enumerate}
\usepackage{verbatim}
\usepackage{multirow}
\usepackage{pgf,tikz}
\usetikzlibrary{arrows}

\CompileMatrices
\newtheorem{introthm}{Theorem}
\newtheorem{introcor}{Corollary}

\newtheorem{theorem}{Theorem}[section]
\newtheorem{lemma}[theorem]{Lemma}

\newtheorem{proposition}[theorem]{Proposition}
\newtheorem{corollary}[theorem]{Corollary}
\theoremstyle{definition}
\newtheorem{definition}[theorem]{Definition}
\newtheorem{example}[theorem]{Example}

\def\cc{{\mathbb C}}

\def\rr{{\mathbb R}}

\def\pp{{\mathbb P}}

\usepackage{tikz}
\usetikzlibrary{matrix}

\begin{document}

\title{Automorphism groups of pseudoreal Riemann surfaces}

\author{Michela Artebani}
\address{
Departamento de Matem\'atica, \newline
Universidad de Concepci\'on, \newline
Casilla 160-C,
Concepci\'on, Chile}
\email{martebani@udec.cl}

\author{Sa\'ul Quispe}
\address{
Departamento de Matem\'atica y Estad\'istica,  \newline
Universidad de La Frontera,  \newline
Casilla 54-D,   
Temuco, Chile}
\email{saul.quispe@ufrontera.cl}

\author{Cristian Reyes}
\address{
Departamento de Matem\'atica, \newline
Universidad de Concepci\'on, \newline
Casilla 160-C,
Concepci\'on, Chile}
\email{creyesm@udec.cl}

\thanks{The authors have been supported 
by Proyectos FONDECYT Regular N. 1130572, 1160897, 3140050 and
Proyecto Anillo CONICYT PIA  ACT1415.
The third author has been supported by CONICYT PCHA/Mag\'isterNacional/2014/22140855.}

\begin{abstract}
A smooth complex projective curve is called pseudoreal if it is isomorphic to its conjugate but is not definable over the reals.
Such curves, together with real Riemann surfaces, form the real locus of the moduli space $\mathcal M_g$.
This paper deals with the classification of pseudoreal curves according to the structure of their automorphism group.
We follow two different approaches existing in the literature: one coming from number theory, 
dealing more generally with fields of moduli of projective curves, and the other from complex geometry, 
through the theory of NEC groups. 
Using the first approach, we prove that the automorphism group of a pseudoreal Riemann surface $X$ is abelian if  
$X/Z({\rm Aut}(X))$ has genus zero, where $Z({\rm Aut}(X))$ is the center of ${\rm Aut}(X)$. 
This includes the case of $p$-gonal Riemann surfaces, already known by results of Huggins 
and  Kontogeorgis. 
By means of the second approach and of elementary properties of group extensions, 
we show that $X$ is not pseudoreal if the center of $G={\rm Aut}(X)$ is trivial and either ${\rm Out}(G)$ contains no involutions 
or ${\rm Inn}(G)$ has a group complement in ${\rm Aut}(G)$. This extends and gives an elementary proof (over $\cc)$ 
of a result by D\`ebes and Emsalem. 
Finally, we provide an algorithm, implemented in MAGMA, which classifies the automorphism groups 
of pseudoreal Riemann surfaces of genus $g\geq 2$, once a list of all groups acting for such genus, 
with their signature and generating vectors, are given.
This program, together with the database provided by J. Paulhus in \cite{Pau15}, 
allowed us to classifiy pseudoreal Riemann surfaces up to genus $10$, extending previous results 
by Bujalance, Conder and Costa.
\end{abstract}
\maketitle

\tableofcontents

\section*{Introduction}

Let $X\subset \mathbb{P}_{\mathbb{C}}^n$ be a smooth complex projective curve 
defined as the zero locus of homogeneous polynomials $p_1,\dots,p_r\in \mathbb{C}[x_0,\dots,x_n]$ 
and let $\overline{X} $ be its conjugate, i.e. 
the zero locus of the polynomials obtained conjugating the coefficients of the polynomials $p_i$. 
The curve $X$ is called {\em pseudoreal} if it is isomorphic to $\overline{X}$ but is not isomorphic to a curve defined 
by polynomials with real coefficients. 
Since any compact Riemann surface can be embedded in projective space as a smooth complex curve, and 
the definition only depends on the isomorphism class of the curve, this allows to define 
the concept of pseudoreal compact Riemann surface.
Equivalently, pseudoreal Riemann surfaces can be defined as Riemann surfaces 
carrying anticonformal automorphisms but no anticonformal involutions.
The disjoint union of the locus of pseudoreal Riemann surfaces with the locus of real Riemann surfaces
is the fixed locus of the natural involution $X\to \overline X$ on the moduli space of curves $\mathcal M_g$.

 In literature, one can find two main approaches to the study of pseudoreal Riemann surfaces: 
 a number-theoretical approach and an approach through NEC groups. 
 The first approach deals, more generally, with the problem of deciding whether the field of moduli of a curve  is a field of definition. In this setting, pseudoreal curves are complex curves having $\mathbb{R}$ as field of moduli, but not  as a field of definition. 
A fundamental tool in this approach is a classical theorem by A. Weil (Theorem \ref{T1}), which provides a necessary and sufficient condition for a projective variety defined over a field $L$, to be definable over a subfield $K\subseteq L$ when the extension is Galois. More recently, P. D\`ebes and M. Emsalem proved that $X/{\rm Aut}(X)$ can be always defined over the field of moduli of $X$ and that $X$ has the same property when a suitable model $B$ of $X/{\rm Aut}(X)$ over the subfield $K\subseteq L$ has a $K$-rational point (see \cite[Corollary 4.3 (c)]{DebEms99}). 
Such result  has been applied by B. Huggins to complete the classification of pseudoreal hyperelliptic curves  \cite[Proposition 5.0.5]{Hugg05} and later by A. Kontogeorgis in \cite{Kon09} in the case of $p$-gonal curves. Unfortunately, the result of D\`ebes-Emsalem is not easy to apply as soon as the genus of $X\slash{\rm Aut}(X)$ is bigger than zero.

A second approach, specific of compact Riemann surfaces, is through the theory of Fuchsian groups, and more generally of {\em non-euclidean crystallographic (NEC) groups}, which are discrete subgroups $\Delta$ of the full automorphism group of the hyperbolic plane $\mathbb{H}$ such that $\mathbb{H}/\Delta$ is a compact Klein surface. 
It follows from the uniformization theorem that giving a compact Riemann surface $X$ of genus $g\geq 2$  containing a group isomorphic to $G$ in its full automorphism group is equivalent to give an epimorphism $\varphi:\Gamma\to G$, where $\Gamma$ is a NEC group, such that $\ker(\varphi)$ is a torsion free Fuchsian group. 
The automorphism group ${\rm Aut}(X)$ of $X=\mathbb{H}/\ker(\varphi)$ will then contain $\varphi(\Gamma^+)=G^+$,
where $\Gamma^+$ is the canonical Fuchsian subgroup of $\Gamma$.
The Riemann surface $X$ is pseudoreal if and only if ${\rm Aut}(X)$ is an index two subgroup of its full automorphism group ${\rm Aut}(X)^{\pm}$ such that ${\rm Aut}(X)^{\pm}\backslash {\rm Aut}(X)$ contains no involutions. 
This idea allowed D. Singerman to prove the existence of pseudoreal Riemann surfaces of any genus $g>1$ (Theorem \ref{existencia}). Moreover, it has been used by E. Bujalance, M. Conder and A.F. Costa in \cite{BuConCo10} and \cite{BuCo14} to classify the full automorphism groups of pseudoreal Riemann surfaces up to genus $4$.

 The aim of this paper is to provide an introduction to both approaches and to show some new results on the classification of automorphism groups of pseudoreal Riemann surfaces.  The paper is organized as follows. 
 
 In section $1$ we deal with the number theoretical approach. We first provide the background material, 
 defining the concepts of field of moduli and  of pseudoreal curve. 
 Moreover, we introduce the theorem by D\`ebes-Emsalem and some consequences of it 
 when $X\slash{\rm Aut}(X)$ has genus $0$.  In particular we prove the following result.

\begin{introthm}[Theorem \ref{cen}] Let $K$ be an infinite perfect field of characteristic $p\neq 2$ and let $F$ be an algebraic closure of $K$. Let $X$ be a smooth projective algebraic curve of genus $g\ge 2$ defined over $F$, $H$ the center of ${\rm Aut}(X)$ and assume that $X\slash H$ has genus $0$. 
If ${\rm Aut}(X)/H$ is neither trivial nor cyclic (if $p=0$) or cyclic of order relatively prime to $p$ (if $p> 0$) 
then $X$ can be defined over its field of moduli relative to the extension $F\slash K$.
\end{introthm}

This immediately implies the following. 
\begin{introcor}[Corollary \ref{cencor}] 
Let $K$ be a  field of characteristic $0$ and let $F$ be an algebraic closure of $K$. 
Let $X$ be a smooth projective algebraic curve of genus $g\ge 2$ defined over $F$ such that $X\slash Z({\rm Aut}(X))$ has genus $0$ 
(where $Z({\rm Aut}(X))$ is the center of  ${\rm Aut}(X)$).
If $X$ can not be defined over its field of moduli $M_{F/K}(X)$, then ${\rm Aut}(X)$ is abelian.
\end{introcor}

The second section considers the NEC group approach. We recall the proof of the existence of pseudoreal Riemann surfaces in every genus 
and a recent result by C.~Baginski  and G.~Gromadzki \cite{BaGo10} which characterizes groups appearing as full automorphisms groups of pseudoreal Riemann surfaces.
In this setting, studying group extensions of degree two, we provide the following result, which gives an alternative 
proof of \cite[Corollary 4.3 (b)]{DebEms99} 
for the extension $\cc/\rr$ and new conditions on the automorphism group of a pseudoreal Riemann surface.

\begin{introthm}[Corollary \ref{comp}] Let $G$ be the conformal automorphism group of a Riemann surface $X$. 
Suppose that $Z(G)=\{1\}$ and that either ${\rm Out}(G)$ has no involutions or ${\rm Inn}(G)$ has group complement in ${\rm Aut}(G)$. 
Then $X$ is not pseudoreal.
\end{introthm}

Finally, we consider the maximal full automorphism groups of pseudoreal Riemann surfaces (see Theorem \ref{maximalll}) 
and we prove the following result.

\begin{introthm}[Theorem \ref{abeliann}] If a pseudoreal Riemann surface $X$ has maximal full automorphism group, 
then its conformal automorphism group is not abelian and $X/Z({\rm Aut}(X))$ has positive genus. \end{introthm}

Section $3$ is about the classification of pseudoreal Riemann surfaces of low genus. 
We develop an algorithm which allows to classify pseudoreal Riemann surfaces of a given genus $g\geq 2$
starting from the list of all groups acting on Riemann surfaces of that genus.
The algorithm has been implemented in a program written for Magma \cite{Magma} based on a program by J. Paulhus \cite{Pau15}.
By means of this program, we are able to provide the classification of pseudoreal Riemann surfaces up to genus $10$, extending previous results 
in \cite{BuConCo10, BuCo14}.

\begin{introthm}[Theorem \ref{class}]
Two finite groups $G$ and $\bar G$ are the conformal and full automorphism groups of a pseudoreal Riemann surface $X$ of genus 
$5 \leq g\leq 10$ if and only if $G = {\rm Aut}(X)$ and $\overline G = {\rm Aut}^{\pm}(X)$ appear in the corresponding table by genus 
among Tables \ref{table5}, \ref{table6}, \ref{table7}, \ref{table8}, \ref{table9} and \ref{table10}.
\end{introthm}
  
\noindent {\em Aknowledgments.} 
This paper grew out of the Master Thesis of Cristian Reyes Monsalve, 
defended in October 2016 at Universidad de Concepci\'on (Chile).
We thank Andrea Tironi for several useful discussions and careful reading,
Antonio Laface for inspiring conversations and for his help in the design of the Magma programs, 
Rub\'en Hidalgo, Jeroen Sijsling and Xavier Vidaux for key remarks which helped to improve the final version.
We especially thank Jennifer Paulhus for her kind assistance in the use of her program.

\section{The number theoretical approach}

\subsection{Fields of moduli of projective curves}

\noindent Everytime we say \emph{curve} we mean a smooth projective algebraic curve.

\begin{definition} \label{fielddefinition} Let $F$ be a field and let $X\subseteq\mathbb{P}^n_{F}$ 
be a curve defined as the zero locus of  homogeneous polynomials  with coefficients in $F$.
If $K\subset F$ is a subfield, we say that $X$  \emph{can be defined over} $K$,
or that $K$ is a \emph{field of definition} of $X$,
  if there exists a curve $Y\subseteq\mathbb{P}^m_{F}$ defined by polynomials with coefficients in $K$ and isomorphic to $X$ over $F$.

 If $f:X\longrightarrow Y$ is a morphism between two curves $X$ and $Y$ defined over a field $F$, 
 we say that \emph{$f$ is defined over $F$} if the polynomials defining $f$ have all their coefficients in $F$.
\end{definition}

If $F\slash K$ is a field extension, the group of automorphisms 
\[
{\rm Aut}(F\slash K):=\{\sigma\in {\rm Aut}(F):\ \sigma|_{K}={\rm id}_K\}  
\]
naturally acts on both curves and morphisms defined over $F$ in the following way.

\begin{definition} Let $F\slash K$ be a field extension and $\sigma\in {\rm Aut}(F\slash K)$.
\begin{enumerate}
\item Given a curve $X=Z(p_1,\dots,p_r)\subseteq\mathbb{P}^n_F$, where $p_1,\dots,p_r\in F[x_0,\dots,x_n]$ 
are homogeneous, let  $X^{\sigma}:=Z(p_1^{\sigma},\dots, p_r^{\sigma})$,  where 
$p_i^{\sigma}$ is obtained applying $\sigma$ to the coefficients of $p_i$. 
\item Given a morphism $f:X\rightarrow Y$ between curves defined over $F$, let
\[
f^{\sigma}:=\sigma\circ f\circ \sigma^{-1}:X^{\sigma}\rightarrow Y^{\sigma},
\] 
where $\sigma:Z\to Z^{\sigma}$ sends $[z_0:\ldots :z_n]$ to 
$[\sigma(z_0):\ldots :\sigma(z_n)]$.
\end{enumerate}
\end{definition}

\noindent A fundamental result in this area is the following

\begin{theorem}\label{T1} \textbf{(Weil's Theorem)} \cite[Theorem 1]{Weil56} Let $X$ be a curve defined over $F$ and let $F\slash K$ be a Galois extension. If for every $\rho\in {\rm Aut}(F\slash K)$ there exists an isomorphism $f_{\rho}:X\longrightarrow X^{\rho}$ defined over $F$ such that $$f_{\sigma\tau}=f_{\tau}^{\sigma}\circ f_{\sigma},\ \ \ \ \forall \sigma,\tau\in {\rm Aut}(F\slash K),$$ then there exist a curve $Y$ defined over $K$ and an isomorphism $g:X\rightarrow Y$ defined over $F$ such that   $g^{\rho}\circ f_{\rho}=g, \forall \rho\in {\rm Aut}(F\slash K)$.
\end{theorem}

 It is natural to ask for the smallest field of definition of a curve, this leads to the concept of field of moduli.

\begin{definition}\label{fieldmoduli}
The \emph{field of moduli} of a curve $X$ defined over $F$ {\em relative to a Galois extension} $F\slash K$ is 
$$M_{F\slash K}(X):={\rm Fix}(F_K(X))\  \text{ where }\ F_K(X)=\{\sigma\in {\rm Aut}(F\slash K) : X\cong_{\overline{F}} X^{\sigma}\}.$$
\end{definition}

If $F\slash K$ is a Galois extension, then it can be easily proved that $M_{F\slash K}(X)\subseteq F'$ for any 
field of definition $F'$ of $X$ such that $K\subseteq F'\subseteq F$.
Moreover, it is clear that the relative fields of moduli of two isomorphic curves over $\overline F$ are isomorphic.

The following is an easy consequence of Weil's theorem together with the fact that $M_{F\slash R}(X)=R$, where $R=M_{F\slash K}(X)$ 
\cite[Proposition 2.1]{DebEms99}.

\begin{proposition}\label{trivialll} If $X$ is a curve defined over a field $F$, $F\slash K$ is a Galois extension and ${\rm Aut}(X)$ is trivial, then $X$ can be defined over its field of moduli $M_{F\slash K}(X)$.
\end{proposition}

Since the generic curve of genus $g>2$ has trivial automorphism group (see \cite[Theorem 2]{Gre63}), by Proposition \ref{trivialll}  we deduce that $X$ is always defined over its field of moduli relative to a Galois extension. 
Moreover, for curves of genus $0$ and $1$,  it is  known that the field of moduli  is a field of definition (see Example \ref{g1}). 
By a result of H. Hammer and F. Herrlich any curve can be defined over a finite extension of its field of moduli \cite{HamHerr03}.
However there are examples of curves of any genus $g\geq 2$ with non-trivial automorphism groups whose field of moduli (relative to a certain field extension) is not a field of definition.
As Weil's theorem suggests, the structure of the automorphism group plays a key role in the definability problem over 
the field of moduli.

 \begin{example} \label{g1}
Let $X$ be a curve of genus $1$ defined over an algebraically closed field $F$ of  characteristic $p\not=2$ and 
let $F/K$ be a Galois extension.
It is well known that $X\cong_{F} C_{\lambda}$, where $C_{\lambda}\subseteq \mathbb{P}_{F}^2$ is the plane cubic defined by 
$$x_1^2x_2-x_0(x_0-x_2)(x_0-\lambda x_2)=0,$$ 
and $\lambda\in F-\{0,1\}$. Moreover $C_{\lambda}$ and $C_{\mu}$ are isomorphic if and only if $j(\lambda)=j(\mu)$, where $j$ is the $j$-invariant  \cite[Proposition 1.7, Ch. III]{Silv}. Given $\sigma\in {\rm Aut}(F/K)$ we have that $C_{\lambda}^{\sigma}=C_{\sigma(\lambda)}$.
Thus 
$$M_{F\slash K}(X):={\rm Fix}(\{\sigma\in{\rm Aut}(F\slash K): j(\lambda)=j(\sigma(\lambda))=\sigma(j(\lambda))\})=K(j(\lambda)).$$ 
For every $\lambda\in F$, one can find a smooth plane model for $C_{\lambda}$ which is defined over $K(j(\lambda))$ \cite[Proposition 1.4]{Silv}, 
so $M_{F\slash K}(X)$ is also a field of definition of $X$.
 \end{example}
 
  In this paper we will focus on the case of the field extension $\cc/\rr$. More precisely, we are interested in the following concept. 
  
\begin{definition} \label{pseudocurve} A \emph{pseudoreal curve} is a complex curve $X$ such that $M_{\mathbb{C}\slash\mathbb{R}}(X)=\mathbb{R}$ 
but $\mathbb{R}$ is not a field of definition for $X$.
\end{definition}

\subsection{D\`ebes-Emsalem theorem}
 
 We recall a very useful theorem, based on Weil's theorem, that gives sufficient conditions for the definability.
 Let $F\slash K$ be a Galois extension, $X$ be a smooth projective curve of genus $g\ge 2$ defined over $F$
  and assume that $K$ is the field of moduli $M_{F\slash K}(X)$.
 Then for every $\sigma\in {\rm Aut}(F\slash K)$ there is an isomorphism $f_{\sigma}:X\longrightarrow X^{\sigma}$.
This induces an isomorphism 
\[
\varphi_{\sigma}:X\slash {\rm Aut}(X)\longrightarrow X^{\sigma}\slash {\rm Aut}(X^{\sigma})
\]
that makes the following diagram commute 
\begin{center}
\begin{tikzpicture}
  \matrix (m) [matrix of math nodes,row sep=3em,column sep=4em,minimum width=2em]
  {
     X & X^{\sigma} \\
     X\slash {\rm Aut}(X) & X^{\sigma}\slash {\rm Aut}(X^{\sigma}). \\};
  \path[-stealth]
    (m-1-1) edge node [left] {$\pi$} (m-2-1)
            edge node [above] {$f_{\sigma}$} (m-1-2)
    (m-2-1.east|-m-2-2) edge node [below] {$\varphi_{\sigma}$}
            node [above] {} (m-2-2)
    (m-1-2) edge node [right] {$\pi^{\sigma}$} (m-2-2);
\end{tikzpicture}
\end{center}

 Composing $\varphi_{\sigma}$ with the canonical isomorphism $X^{\sigma}\slash {\rm Aut}(X^{\sigma})\longrightarrow (X\slash {\rm Aut}(X))^{\sigma}$ this gives  a family of isomorphisms $$\{\overline{\varphi}_{\sigma}:X\slash {\rm Aut}(X)\longrightarrow (X\slash {\rm Aut}(X))^{\sigma}\}_{\sigma\in {\rm Aut}(F\slash K)}$$ which satisfies the condition in Weil's Theorem. 
Thus there exists a curve $B$ isomorphic to $X\slash {\rm Aut}(X)$ over $F$ which is defined over $K$. 
This curve $B$ is the so called \emph{canonical K-model} for $X\slash {\rm Aut}(X)$.

\begin{theorem}\label{DE} \cite[Corollary 4.3, (c)]{DebEms99} Let $F\slash K$ be a Galois extension and 
$X$ be a smooth projective curve of genus $g\ge 2$ defined over $F$ such that 
the order of ${\rm Aut}(X)$ is not divisible by the characteristic of $K$.
Assume that $K$ is the field of moduli of $X$ and that $F$ is the separable closure of $K$.
If the canonical $K$-model $B$ of $X\slash {\rm Aut}(X)$ has at least one $K$-rational point off the branch point set of the natural quotient $\pi:X\to B$,
then $K$ is a field of definition of $X$.
\end{theorem}

 \subsection{The case when $X/{\rm Aut}(X)$ has genus zero}
 
 In case the field has characteristic zero and the $K$-canonical model $B$ of $X/{\rm Aut}(X)$ has genus zero, 
 the existence of one $K$-rational point in $B$ implies that $B\cong_K \pp^1_K$, so there are infinitely many of them. 
 Moreover, the existence of a $K$-rational point is implied by the existence of a divisor of odd degree defined 
 over $K$ \cite[Lemma 5.1]{Hugg04}. These ideas led to the following beautiful result by B. Huggins.

\begin{theorem} \cite{Hugg04} \label{hug} Let $K$ be a perfect field of characteristic $p\neq 2$ and let $F$ be an algebraic closure of $K$. Let $X$ be a hyperelliptic curve defined over $F$ and let $G={\rm Aut}(X)\slash\langle i\rangle$, where $i$ is the hyperelliptic involution of $X$. If $G$ is not cyclic, or it is cyclic of order not divisible by $p$, then $X$ can be defined over its field of moduli relative to the extension $F\slash K$.
\end{theorem}

The proof relies on the fact that  the quotient group ${\rm Aut}(X)\slash\langle i\rangle$ acts on $\pp_F^1$, and thus 
it is isomorphic to a finite subgroup of ${\rm PGL}(2,F)$ (when $p=0$ this is isomorphic  to either  $C_n, D_n, A_4,A_5$ or $S_4$  \cite{MaVa80}).
The author showed that in all cases, except in the cyclic case, one can find a $K$-rational point in the canonical model $B$ 
of $X\slash {\rm Aut}(X)$ by studying explicitely the action of ${\rm Aut}(F/K)$ on the function field of $B$.

The case of hyperelliptic curves such that  ${\rm Aut}(X)\slash\langle i\rangle$ is cyclic of order coprime to the characteristic of the 
ground field has been completely described in \cite{LRS13}.

In  \cite{Kon09} Kontogeorgis generalized Theorem \ref{hug} considering certain $p$-cyclic covers of the projective line,
where $p$ is prime, in any characteristic. In particular he proved the analogous of Huggins' theorem in case the subgroup generated by 
the cyclic automorphism of order $p$ is normal in ${\rm Aut}(X)$.
Recently, this has been further  generalized by R.~Hidalgo and S.~Quispe \cite{HidQui15} by considering 
some particular subgroups of the automorphisms group of the curve, defined as follows.

\begin{definition} Let $X$ be a curve. A subgroup $H\subseteq {\rm Aut}(X)$ is \emph{unique up to conjugation} 
if any subgroup $K\subseteq {\rm Aut}(X)$ isomorphic to $H$ such that the signatures of the covers 
$\pi_H:X\longrightarrow X\slash H$ and $\pi_K:X\longrightarrow X\slash K$ are the same is conjugated to $H$.
\end{definition}

Observe that the subgroup $H$ generated by the hyperelliptic involution or by a cyclic automorphism of prime order $p$ 
with quotient $X/H$ of genus zero are unique up to conjugation. In the following, we denote by $N_G(H)$ the normalizer of 
a subgroup $H$ of $G$.

\begin{theorem} \cite[Theorem 1.2]{HidQui15} Let $K$ be an infinite perfect field of characteristic $p\neq 2$ and let $F$ be an algebraic closure of $K$. Let $X$ be a smooth projective algebraic curve of genus $g\ge 2$ defined over $F$ and let $H$ be a subgroup of ${\rm Aut}(X)$ which is unique under conjugation, so that $X\slash H$ has genus $0$. If $N_{{\rm Aut}(X)}(H)\slash H$ is neither trivial nor cyclic (if $p=0$) or cyclic of order relatively prime to $p$ (if $p> 0$) then $X$ can be defined over its field of moduli relative to the extension $F\slash K$.
\end{theorem}

A similar proof led us to prove the following result, where $H$ is replaced by the center of the automorphism group.
Observe that the center is not unique up to conjugation in general (see Example \ref{cu}).

\begin{theorem}\label{cen}  Let $K$ be an infinite perfect field of characteristic $p\neq 2$ and let $F$ be an algebraic closure of $K$. Let $X$ be a smooth projective algebraic curve of genus $g\ge 2$ defined over $F$, $H$ the center of ${\rm Aut}(X)$ and assume that $X\slash H$ has genus $0$. 
If ${\rm Aut}(X)/H$ is neither trivial nor cyclic (if $p=0$) or cyclic of order relatively prime to $p$ (if $p> 0$) 
then $X$ can be defined over its field of moduli relative to the extension $F\slash K$.
\end{theorem}

\begin{proof}
We will prove the theorem for the case $p=0$, since the proof for the  case $p\neq 0$ is the same as in \cite[Theorem 1.2]{HidQui15}.
Without loss of generality we can assume $K=M_{F\slash K}(X)$ (see \cite[Proposition 2.1]{DebEms99}). 
Given  $\sigma\in {\rm Aut}(F\slash K)$, there exists an isomorphism $f_{\sigma}:X\longrightarrow X^{\sigma}$. 
We start observing the following, where $G={\rm Aut}(X)$: \\

\noindent \emph{Claim 1.} $f_{\sigma}Z(G)f_{\sigma}^{-1}=Z(G)^{\sigma}.$ 
 For every $a\in Z(G)$ and $b\in G^{\sigma}$ we have $$(f_{\sigma}af_{\sigma}^{-1})b=f_{\sigma}a(\underbrace{f_{\sigma}^{-1}bf_{\sigma}}_{b'\in G})f_{\sigma}^{-1}=f_{\sigma}ab'f_{\sigma}^{-1}\underbrace{=}_{a\in Z(G)}f_{\sigma}b'af_{\sigma}^{-1}=b(f_{\sigma}af_{\sigma}^{-1}).$$ 
 This says that $f_{\sigma}Z(G)f_{\sigma}^{-1}\subseteq Z(G^{\sigma})$, which is equal to $Z(G)^{\sigma}$. The other inclusion is obtained analogously. 
 By  Claim 1 and the fact that $f_{\sigma}Gf_{\sigma}^{-1}=G^{\sigma}$, there exist two isomorphisms $g_{\sigma}$ and $h_{\sigma}$ such that the following diagram commutes:

 \[
\xymatrix{
X\ar[r]^{f_{\sigma}}\ar[d]_{\pi_1} & X^{\sigma}\ar[d]^{\pi_1^{\sigma}}\\
X/Z(G)\ar[r]^{g_{\sigma}}\ar[d]_{\pi_2} & (X/Z(G))^{\sigma}\ar[d]^{\pi_2^{\sigma}} \\
X/G\ar[r]^{h_{\sigma}} & (X/G)^{\sigma}
}
\]

\noindent \emph{Claim 2.} Without loss of generality, we can assume that the branch locus of $\pi_2$ in $X/G\cong \pp^1_F$ is $\mathcal{B}=\{0,1,\infty\}$.

\medskip

\noindent 
The covering group of $\pi_2$ is isomorphic to $ G\slash Z(G)$ and it is a finite group acting on the projective line. By our hypothesis, $G\slash Z(G)$ is not a cyclic group, so it must be isomorphic to either $D_{n}, A_4, A_5$ or $S_4$  \cite{MaVa80}. In any of these four cases the branch locus of $\pi_2$ consists of $3$ points, which can be taken to be $0,1,\infty$ up to a projectivity. \\

\noindent \emph{Claim 3.} There exists an isomorphism $R:X\slash G\longrightarrow B$ to the canonical $K$-model $B$ of $X\slash G$ such that $R=R^{\sigma}\circ h_{\sigma}$.

\medskip

We will prove that given $\sigma\in {\rm Aut}(F/K)$ there exists a unique isomorphism $h_{\sigma}$ as in the diagram. 
This implies that $\{h_{\sigma}\}_{\sigma\in{\rm Aut}(F\slash K)}$ satisfies Weil's Theorem \ref{T1} and gives the statement.

 Let $S:=\pi_2\circ \pi_1$. Assume that there exist other isomorphisms $f_{\sigma}'$ and $h_{\sigma}'$ such that $h_{\sigma}'\circ S=S^{\sigma}\circ f_{\sigma}'$. Then $f_{\sigma}^{-1}\circ f_{\sigma}'=g\in G$ so  we have 
$$h_{\sigma}'\circ S=S^{\sigma}\circ f_{\sigma}'=S^{\sigma}\circ f_{\sigma}\circ g=S^{\sigma}\circ g'\circ f_{\sigma}=S^{\sigma}\circ f_{\sigma}=h_{\sigma}\circ S,$$ 
where in the third equality $g':=f_{\sigma}\circ g\circ f_{\sigma}^{-1}\in G^{\sigma}$  and the fourth equality follows from the fact that $S^{\sigma}$ is the quotient by $G^{\sigma}$. Thus $h_{\sigma}'=h_{\sigma}$, proving that $h_{\sigma}$ is uniquely determined by $\sigma$. 
By Weil's Theorem \ref{T1} we obtain that there exists an isomorphism $R:X\slash G\longrightarrow B$ such that $B$ is a curve of genus $0$ defined over $K$ and we have the following commutative diagram $(*)$:

\[
\xymatrix{
X\ar[rr]^{f_{\sigma}}\ar[d]_S & & X^{\sigma}\ar[d]^{S^{\sigma}}\\
X/G\ar[rr]^{h_{\sigma}} \ar[rd]_R& & (X/G)^{\sigma}\ar[ld]^{R^{\sigma}} \\
& B & 
}
\]

 Since $g_{\sigma}$ and $h_{\sigma}$ are isomorphisms, we have that $h_{\sigma}(\mathcal{B})=\sigma(\mathcal{B})=\mathcal B\ (**)$. \\

\noindent \emph{Claim 4.} $B$ has a $K$-rational point $r$ outside the branch locus of $R\circ S$.\\

\noindent The branch divisor $D=R(0)+R(1)+R(\infty)$ satisfies 
$$D^{\sigma}=\sigma(R(0))+\sigma(R(1))+\sigma(R(\infty))=R^{\sigma}(\sigma(0))+R^{\sigma}(\sigma(1))+R^{\sigma}(\sigma(\infty))$$ 
$$=R^{\sigma}(0)+R^{\sigma}(1)+R^{\sigma}(\infty)\overset{(*)}{=}R\circ h_{\sigma}^{-1}(0)+R\circ h_{\sigma}^{-1}(1)+R\circ h_{\sigma}^{-1}(\infty)\overset{(**)}{=}D,$$ so $B$ has a $K$-rational divisor of degree $3$. By  \cite[Lemma 5.1]{Hugg04} this implies that $B$ has a $K$-rational point, and thus $B$ is isomorphic to $\pp^1_K$ over $K$. In particular, $B$ has a $K$-rational point $r$ outside the branch locus of $R\circ S$. 

\medskip

 By D\`ebes-Emsalem theorem (Theorem \ref{DE}\,(b)) and \emph{Claim 4} we conclude that $K$ is a field of definition of $X$.
\end{proof}

A group over its center is cyclic if and only if the group is abelian, thus we obtain the following.

\begin{corollary}\label{cencor}
Let $K$ be a  field of characteristic $0$ and let $F$ be an algebraic closure of $K$. 
Let $X$ be a smooth projective algebraic curve of genus $g\ge 2$ defined over $F$ such that $X\slash Z({\rm Aut}(X))$ has genus $0$. 
If $X$ can not be defined over its field of moduli $M_{F/K}(X)$, then ${\rm Aut}(X)$ is abelian.
\end{corollary}

\begin{example}\label{cu}
Consider the plane quartic $X$ defined by $x^4+y^4+z^4+ax^2y^2+bxyz^2=0$, where 
 $ab\not=0$. By \cite[Theorem 29]{Bars05} $X$ has automorphism group isomorphic to $D_4$ and generated by
$$f: [x:y:z]\mapsto [y:x:z],\ g:[x:y:z]\mapsto [ix:-iy:z].$$ 
The map $g^2$ has order $2$ and generates the center of ${\rm Aut}(X)$.
The group $\langle f\rangle$ also has order $2$, and the quotients 
$X\slash\langle f\rangle$ and $X\slash\langle g^2\rangle$ both 
have signature $(1;[2,2,2,2])$. 
Thus $Z({\rm Aut}(X))$ is not unique up to conjugation in this case.
\end{example}

\subsection{Odd signature curves}

 We present here another easy criterion for the definability of a curve $X$ over its field of moduli in terms of 
 the signature of the covering $\pi:X\longrightarrow X\slash {\rm Aut}(X)$ provided that $X\slash {\rm Aut}(X)$ has genus $0$.

\begin{definition} 
A smooth projective curve $X$ of genus $g\geq 2$ has {\em odd signature} if the signature of the covering $\pi: X \to X/{\rm Aut}(X)$ 
is of the form $(0; c_1,\dots,c_r)$ where some $c_i$ appears exactly an odd number of times.
\end{definition}

The following result can be proved by means of D\`ebes-Emsalem theorem and applying  \cite[Lemma 5.1]{Hugg04} to the branch locus of  the projection $\pi:X\longrightarrow X\slash {\rm Aut}(X)$.  In case the field is finite, the result follows from \cite[Corollary 2.11]{Hugg04}.

\begin{theorem} \cite[Theorem 2.5]{ArtQui12} \label{AQ} Let $X$ be a smooth projective curve of genus $g\ge 2$ 
defined over an algebraically closed field $F$, and let $F\slash K$ be a Galois extension. If $H$ is a subgroup of ${\rm Aut}(X)$ unique up to conjugation and $\pi_N:X\longrightarrow X\slash N_{{\rm Aut}(X)}(H)$ is an odd signature cover, then $M_{F\slash K}(X)$ is a field of definition for $X$.
\end{theorem}

The previous theorem immediately implies that $X$ can be defined over its field of moduli anytime 
the branch locus of $\pi:X\to X/{\rm Aut}(X)$ has odd cardinality. In \cite[Corollary 3.5]{ArtQui12} it has been applied 
to show that non-normal $p$-gonal curves in characteristic zero, where $p$ is prime, can always be defined over their field of moduli.

\section{Pseudoreal Riemann surfaces and NEC groups}
In this section we will concentrate on complex curves, or equivalently (embedded) compact Riemann surfaces.

\subsection{Pseudoreal Riemann surfaces} Since the concepts of field of moduli and field of definition of a curve only depend on its 
isomorphism class, then we can extend the definition to compact Riemann surfaces after choosing any embedding of them 
in a projective space. We recall that an {\em antiholomorphic (or anticonformal)} morphism between two Riemann 
surfaces $X,Y$ is a continuous map $X\to Y$ whose transition maps composed with the complex conjugation $J$ are holomorphic.

\begin{proposition} A compact Riemann surface $X$ has field of moduli 
$M_{\mathbb{C}\slash\mathbb{R}}(X)=\mathbb{R}$ if and only if it has an anticonformal automorphism, 
and it has field of definition $\mathbb{R}$ if and only if it has an anticonformal involution.
\end{proposition}

\begin{proof} We will identify $X$ with a smooth complex projective curve, after choosing an embedding of it in a projective space $\pp^n_\cc$. 
Note that $M_{\mathbb{C}\slash\mathbb{R}}(X)=\mathbb{R}$ if and only if $X\cong_{\mathbb{C}}\overline{X}$, where we denote by 
$\overline{X}$ the curve $X^{c}$ obtained applying the complex conjugation to the coefficients of the defining polynomials of $X$.
If  $f:\overline{X}\longrightarrow X$ is an isomorphism and $J([x_0:\dots:x_n])=[\bar x_0:\dots:\bar x_n]$, then $f\circ J|_X:X\longrightarrow X$ 
is an anticonformal automorphism. Conversely, if $g$ is 
an anticonformal automorphism of $X$,  then $J\circ g$ is an isomorphism  between $X$ and $\overline{X}$. This proves the first statement.

If $X$ has field of definition $\mathbb{R}$, then we can assume without loss of generality that $X=\overline{X}$, so that
the map $J|_X$ is an anticonformal involution of $X$. 
Conversely assume that $X$ has an anticonformal involution $\tau$. 
If ${\rm Aut}(\mathbb{C}\slash\mathbb{R})=\{e,c\}$, let $f_e:X\longrightarrow X^{e}=X$ be the identity and $f_{c}:=J|_X\circ \tau:X\to \overline X=X^c$. 
We have $$f_{c}=(f_{c})^{e}\circ f_e,\ \ \ \ \ \ f_{c}=(f_e)^{c}\circ f_{c},\ \ \ \ \ \ f_e=(f_e)^{e}\circ f_e,$$ and we see that $$(f_{c})^{c}\circ f_{c}=(J|_X\circ \tau)^{c}\circ (J|_X\circ \tau)=((J|_X)^{c}\circ \tau^{c}\circ J|_X)\circ \tau=\tau\circ \tau={\rm Id}_X=f_e,$$ 
so the collection $\{f_e, f_{c}\}$ satisfies the condition of Weil's theorem. Hence $X$ can be defined over $\mathbb{R}$. \end{proof}

\begin{corollary} A Riemann surface is pseudoreal if and only if it has antiholomorphic automorphisms but no antiholomorphic involution.
\end{corollary}

Observe that with this characterization it is clear that any pseudoreal Riemann surface has non-trivial 
automorphism group.  In what follows we will denote by ${\rm Aut}^{\pm}(X)$ the {\em full automorphism group} 
of a Riemann surface, i.e. the group containing both automorphisms and anti-holomorphic automorphisms of $X$.
To study such group we recall the language of Fuchsian and NEC groups.

\subsection{NEC groups}

\begin{definition} A \emph{NEC group} is a discrete subgroup $\Gamma$ of ${\rm PGL}(2,\mathbb{R})$, the full automorphism group  of $\mathbb{H}=\{z\in\mathbb{C}:\ {\rm Im}(z)>0\}$, such that $\mathbb{H}\slash\Gamma$ is a compact surface. A \emph{Fuchsian group} is a NEC group contained in ${\rm PSL}(2,\mathbb{R})$, the group of conformal automorphisms of $\mathbb{H}$. If $\Gamma$ is a NEC group which is not a Fuchsian group, it is called a \emph{proper NEC group} and the group $\Gamma^+:=\Gamma\cap {\rm PSL}(2,\mathbb{R})$ is called the \emph{canonical Fuchsian subgroup of $\Gamma$}.
\end{definition}

 Every NEC group $\Gamma$ admits a vector, called the \emph{signature of $\Gamma$}, given by 
\begin{equation}\label{ssg} s(\Gamma)=(g;\pm;[m_1,\ldots, m_r],\{C_1,\ldots, C_k\}),\end{equation} where $g$ is the genus of the quotient space $\mathbb{H}\slash\Gamma$; the sign $+$ or $-$ depends on the orientability or non-orientability of the quotient space, respectively; the integers $m_i\ge 2$ are the ramification indices of the $r$ branch points of the quotient $\pi:\mathbb{H}\longrightarrow\mathbb{H}\slash\Gamma$; and the $C_i$ are $s_i$-uples of integers $$C_i=(n_{i1},\ldots, n_{is_i}),$$ such that $n_{ij}\ge 2$, where that values represents the ramification index of the quotient $\mathbb{H}\slash\Gamma$ in its $i$-th boundary component (if it has no border, we consider no $C_i$ in the signature).
 Every NEC group with signature (\ref{ssg}) (for $+$ sign) has a presentation as a group given by the generators $$x_1,\ldots, x_r\ \ \ \ a_1,b_1,\ldots, a_g,b_g\ \ \ \ e_1,\ldots, e_k\ \ \ \ c_{10},\ldots, c_{1s_1},\ldots, c_{k0},\ldots, c_{ks_k}$$ which are called \emph{elliptic, hyperbolic, boundary and reflection elements}, respectively, that satisfy the relations $$x_i^{m_i}=1\ \ \ \forall i\in\{1,\ldots ,r\}$$ 
 $$c_{ij-1}^2=c_{ij}^2=(c_{ij-1}c_{ij})^{n_{ij}}=1,\ \ c_{is_i}=e_i^{-1}c_{i0}e_i\ \ \ \forall i\in\{1,\ldots, k\}\ ,\ \ \forall j\in\{0,\ldots, s_i\}$$ 
 $$x_1\ldots x_re_1\ldots e_k[a_1,b_1]\ldots [a_g,b_g]=1,$$ 
 where $[a_i,b_i]=a_ib_ia_i^{-1}b_i^{-1}$. In case the signature has a $-$ sign, the relations are the same except the last one which is replaced by 
 $$x_1\ldots x_re_1\ldots e_kd_1^2\ldots d_g^2=1,$$ 
 where the $d_i$ are $g$ \emph{glide reflections}, which are antiholomorphic 
 elements of infinite order of the NEC group (for details see \cite[Ch. 0]{BK}).

 By the Riemann surfaces uniformization theorem we know that every compact Riemann surface of genus greater than $1$ is biholomorphic to $\mathbb{H}\slash\Delta$, where $\Delta$ is a Fuchsian group without torsion. 
 It follows with this presentation that 
 $${\rm Aut}(X)\cong N_{{\rm PSL}(2,\mathbb{R})}(\Delta)\slash\Delta,\quad {\rm Aut}^{\pm}(X)\cong N_{{\rm PGL}(2,\mathbb{R})}(\Delta)\slash\Delta,$$
 where $N_{-}(\Delta)$ denotes the normalizer of $\Delta$ in both cases. In fact, any group $H\le {\rm Aut}(X)$ is isomorphic to some quotient group $\Gamma\slash\Delta$, where $\Gamma$ is a Fuchsian group that contains $\Delta$ as a normal subgroup of finite index. 
 In a similar way, a subgroup $H\le {\rm Aut}^{\pm}(X)$ is isomorphic to a quotient $\Gamma\slash\Delta$, where $\Gamma$ is a NEC group (see \cite{BK}).
 
With the previous notation, studying Riemann surfaces where an abstract group $H$ acts as full automorphism group with a certain signature is equivalent to 
study the possible epimorphisms $\theta:\Gamma\rightarrow H$ such that $\ker(\theta)$ is a torsion free group contained in $\Gamma^+$. 
 Observe that $\ker(\theta)$ is a torsion free Fuchsian group if and only if $\theta$ preserves 
the orders of the finite order elements of $\Gamma$ (see \cite[p.~12]{BS}).

The images by $\theta$ of the generators of $\Gamma$ form a vector with entries in $H$ 
which is called {\em generating vector} of $\theta$.

The following theorem  relates the signature of a NEC group with the signature of its canonical Fuchsian subgroup \cite[Theorem 2]{Sin74}.

\begin{theorem} \label{5.1} Let $X\slash \Delta$ be a Riemann surface (considered as a Klein surface) and denote by $\Gamma$ the NEC group $N_{{\rm PGL}(2,\mathbb{R})}(\Delta)$ which corresponds to its full automorphism group. If $\Gamma$ is a proper NEC group which has no reflections, then its signature has the form $$(g;-;[m_1,\ldots ,m_r]),$$ and the signature of his canonical Fuchsian subgroup $\Gamma^{+}$ will be of the form 
$$(g-1;+;[m_1,m_1,\ldots ,m_r,m_r]),$$ 
where every $m_i$ appears two times.
\end{theorem}

\subsection{Finite-extendability of Fuchsian and NEC signatures}

We now recall the concept of finitely maximal Fuchsian signature, since it will be important to decide whether a given group is the complete automorphism 
group of a Riemann surface or not. In what follows we will denote by $s(G)$ the signature of a NEC group $G$ and by $d(s(G))=6g-6+2r$ the real dimension of the Teichm\"uller space of the signature $s(G)=(g;[m_1,\ldots ,m_r])$ (see \cite[p.~19]{Sin74b}).

\begin{definition} A Fuchsian group $\Delta$ is \emph{finitely maximal} if it is not contained properly in another Fuchsian group with finite index. 
 The signature $(g;[m_1,\ldots ,m_r])$ of a Fuchsian group $\Delta$ (a Fuchsian signature) is \emph{finitely maximal} if $d(s(\Gamma))\neq d(s(\Delta))$ for every Fuchsian group $\Gamma$ containing $\Delta$ as a proper subgroup.
\end{definition}

 Almost all Fuchsian signatures are finitely maximal, and those which are not finitely maximal were identified 
by L. Greenberg in  \cite{Gre63} and D. Singerman  in \cite{Sin72}. 
They determined there the so called \emph{Singerman list} (see \cite[p.~19]{BK}), which contains the only $19$ non finitely maximal Fuchsian signatures.

Considering Theorem \ref{5.1} we are interested in some particular signatures from Singerman list.

\begin{table}[!h]
\centering\renewcommand{\arraystretch}{1.1}\setlength{\tabcolsep}{10pt}
\caption{Even signatures from Singerman list}
\label{Sing}
    \begin{tabular}{| c | c | c |}
    \hline
    $\sigma_1$ & $\sigma_2$ & $[\sigma_2:\sigma_1]$ \\ \hline
    $(2;[-])$ & $(0;[2,2,2,2,2,2])$ & $2$ \\ \hline
    $(1;[t,t])$ & $(0;[2,2,2,2,t])$ & $2$ \\ \hline
    $(0;[t,t,t,t])\ t\ge 3$ & $(0;[2,2,2,t])$ & $4$ \\ \hline
    $(0;[t_1,t_1,t_2,t_2])\ t_1+t_2\ge 5$ & $(0;[2,2,t_1,t_2])$ & $2$ \\ \hline
    \end{tabular}

\end{table}

  The analogous of Singerman list for NEC groups was developed and completed in \cite{Buj82} and \cite{EstIzq06}. 
  We will use these lists to find NEC signatures which correspond to full automorphism groups of pseudoreal Riemann surfaces.
 By  \cite[Remark 1.4.7]{BS}, the signature $s(\Gamma)$ of a proper NEC group is finitely maximal if the same holds for the signature 
 $s(\Gamma^+)$ of its canonical Fuchsian subgroup.
 
 In \cite[Section 4]{BaGo10}, the authors studied under which conditions a finite group $G$ with a given non finitely maximal NEC signature can act as the full automorphism group of a pseudoreal Riemann surface. The three NEC signatures they studied were $$(1;-;[k,l];\{-\}),\ \ \ \ \ \ (1;-;[k,k];\{-\}),\ \ \ \ \ \ (2;-;[k];\{-\}),$$ which are associated to the non finitely maximal Fuchsian \emph{even} signatures $$(0;[k,k,l,l]),\ \ \ (0;[k,k,k,k]),\ \ \ (1;[k,k]),$$
of the Singerman list (see Table \ref{Sing} and Theorem \ref{5.1}). 
We study the missing case of the non finitely maximal NEC signature $(3;-;[-];\{-\})$, which will be needed to complete the classification 
of possible automorphism groups for pseudoreal Riemann surfaces. In what follows we say that an automorphism group has an {\em essential} action 
on a Riemann surface if it contains anticonformal elements.

\begin{lemma} \label{BG4} Let $\Delta$ be a NEC group with signature $(3;-;[-];\{-\})$. 
There exists an epimorphism $\theta:\Delta\longrightarrow G$ onto a finite group $G$ defining an essential action  on a pseudoreal Riemann surface if and only if $G$ is a non-split extension of some of its subgroups $H$ of index $2$, $G$ is generated by three elements $d',d'',d'''$ such that $d',d'',d'''\not\in H$, $(d')^2(d'')^2(d''')^2=1$ and the map $$d'\mapsto (d')^{-1},\ \ \ d''\mapsto (d')^2(d'')^{-1}(d')^{-2},\ \ \ d'''\mapsto (d''')^{-1}$$ does not induce an automorphism of $G$. Furthermore, such a group $G$ is necessarily the full automorphism group of a pseudoreal Riemann surface on which it acts.
\end{lemma}

\begin{proof}
Suppose we have an epimorphism $\theta:\Delta\longrightarrow G$ 
onto a finite group $G$ defining an essential action on the pseudoreal Riemann surface 
$X=\mathbb{H}\slash \ker(\theta)$. 
Then $H:=\theta(\Delta^{+})$ is an index $2$ subgroup of $G$, 
because $G$ has anticonformal elements. 
The extension $H\le G$ is non-split since otherwise $G\backslash H$ would contain anticonformal involutions, 
contradicting the fact that $X$ is pseudoreal. 
We have $\Delta=\langle d_1,d_2,d_3:\ d_1^2d_2^2d_3^2=1\rangle$ 
where the $d_i$'s are glide reflections, 
so the anticonformal elements $d':=\theta(d_1), d'':=\theta(d_2)$ and $d''':=\theta(d_3)$ 
can not belong to $H$. 
To prove the statement we need to show that the map 
$$d'\mapsto (d')^{-1},\ \ \ d''\mapsto (d')^2(d'')^{-1}(d')^{-2},\ \ \ d'''\mapsto (d''')^{-1}$$ 
does not induce an automorphism of $G$. 
To see this, observe that by \cite[p.~529-530]{Buj82} there is a NEC group $\Delta'$ 
with the unique signature $(0;+;[2,2,2],\{(-)\})$ containing $\Delta$ as a subgroup of index $2$. 
By \cite[Proposition 4.8]{Buj82} we know that if 
$$\Delta'=\langle x_1,x_2,x_3,c_1,e_1: x_1x_2x_3e_1=1,\ e_1^{-1}c_1e_1c_1=1,\ x_1^2=x_2^2=x_3^2=1,\ c_1^2=1\rangle$$ 
then $\Delta$ can be written as 
$$\Delta=\langle d_1:=c_1x_1,\ d_2:=x_1c_1x_1x_2,\ d_3:=x_2x_1c_1x_1x_2x_3\rangle\le \Delta'.$$ 
If we conjugate every generator of $\Delta$ by $c_1$ we get 
$$c_1^{-1}d_1c_1=d_1^{-1},\ \ \ c_1^{-1}d_2c_1=d_1^2d_2^{-1}d_1^{-2},\ \ \ c_1^{-1}d_3c_1=d_3^{-1},$$ 
so $\ker(\theta)$ is a normal subgroup of $\Delta'$ if and only if the images of $d_1,d_2$ and $d_3$ 
through $\theta$ satisfy that the map 
$$d'\mapsto (d')^{-1},\ \ \ d''\mapsto (d')^2(d'')^{-1}(d')^{-2},\ \ \ d'''\mapsto (d''')^{-1}$$ 
induces an automorphism of $\Delta\slash\ker(\theta)=G$. 
So the assertion follows, since if $\ker(\theta)$ were a normal subgroup of $\Delta'$, then $\Delta'\slash\ker(\theta)\cong {\rm Aut}^{\pm}(X)$ and it 
would contain $c_1\ker(\theta)$, which is an anticonformal involution, contradicting the hypothesis that $X$ is pseudoreal.

Conversely, for a NEC group $\Delta$ with signature $(3;-;[-],\{-\})$ and a non-split extension $H\le G$ of degree $2$, we can consider the map $\theta(d_1)=d',\ \theta(d_2)=d'',\ \theta(d_3)=d'''$ which induces an epimorphism $\theta:\Delta\longrightarrow G$, defining an essential action on $X:=\mathbb{H}\slash \ker(\theta)$. We must have that $G$ is the full automorphism group of $X$, because otherwise $\ker(\theta)$ would be a normal subgroup of a NEC group $\Delta'$ with signature $(0;+;[2,2,2],\{(-)\})$, and so by the previous part of the proof, the mapping $$d'\mapsto (d')^{-1},\ \ \ d''\mapsto (d')^2(d'')^{-1}(d')^{-2},\ \ \ d'''\mapsto (d''')^{-1}$$ would define an automorphism of $G$, contradicting our assumptions. Finally, since $G$ is a non-split extension of $H$, then $G\backslash H$ contains no involutions, thus $X$ is a pseudoreal Riemann surface.
\end{proof}

\subsection{Existence for any genus}

In \cite[Theorem 1 and p.~48]{Sin80} the author proved the following result. We provide the proof since it is an interesting application 
of the approach through NEC groups.

\begin{theorem}\label{existencia} There exist pseudoreal Riemann surfaces for every genus $g\ge 2$.
\end{theorem}

\proof Consider a NEC group $\Delta$ with signature $(1;-;[2^{g+1}];\{-\})$, where $g$ is even. If $x_1,\dots, x_{g+1}$ is the set of elliptic generators and $d_1$ is the glide reflection, which together generate $\Delta$, we can define an epimorphism 
$\theta:\Delta\longrightarrow C_4=\langle a:a^4=1\rangle$ given by 
$$\theta(x_i)=a^2,\ \ \ \forall i\in\{1,\ldots ,g+1\},\ \ \ \ \ \ \theta(d_1)=a.$$ 
Since $\theta$ preserves the orders of the elliptic generators then $\ker(\theta)$ is torsion free (see \cite[p.~12]{BS}), 
so the quotient $X=\mathbb{H}\slash \ker(\theta)$ is a Riemann surface such that ${\rm Aut}^{\pm}(X)$ contains the group $\Delta\slash \ker(\theta)\cong C_4$. Since $\Delta$ has finitely maximal signature \cite{Buj82}, we can conclude  that ${\rm Aut}^{\pm}(X)\cong C_4$. This Riemann surface $X$ has genus $g$ and has anticonformal automorphisms, but no anticonformal involutions, because $a^2$, the only element of order $2$ in $C_4$, is in the conformal part ${\rm Aut}^{+}(X)\cong \{1,a^2\}$. Thus $X$ is pseudoreal.

The proof for odd $g$ is similar, taking the NEC signature $(2;-;[2^{g-1}];\{-\})$ and considering the epimorphism $\theta:\Delta\longrightarrow C_4=\langle a:a^4=1\rangle$ given by $$\theta(x_i)=a^2,\ \ \ \forall i\in\{1,\ldots ,g-1\},\ \ \ \ \ \ \theta(d_1)=\theta(d_2)=a. \eqno\qed$$

\subsection{Full groups of pseudoreal Riemann surfaces and group extensions}

Let $\overline G$ be a group acting on a Riemann surface $X$ and containing antiholomorphic automorphisms 
(i.e. acting as an essential group). Then there is an exact sequence of groups
\begin{equation}\label{ext}
1\longrightarrow G\longrightarrow \overline{G}\longrightarrow C_2\longrightarrow 0,
\end{equation}
where $G$ denotes the subgroup of holomorphic automorphisms in $\bar G$.
If $X$ is pseudoreal, then $\overline G\setminus G$ contains no order two elements, or equivalently the sequence is non-split.
On the other hand, the group $G$ contains an element of order two by Cauchy's theorem applied to $\overline G$, thus  
$G$ has even order.
In \cite{BaGo10} the authors proved that these conditions are sufficient for a group $\overline G$ to act 
on a pseudoreal Riemann surface.

\begin{theorem} \cite[Theorem 3.3]{BaGo10}\label{bagofull} A finite group $\overline{G}$ acts as an essential group on a 
pseudoreal Riemann surface $X$ if and only if it is a non-split extension of a group of even order by the cyclic group of order $2$. For any such group, 
there exists a Riemann surface $X$ having $\overline{G}$ as its full automorphism group.
\end{theorem}

As a consequence of the previous theorem,  it can be easily seen that no symmetric or dihedral group 
can be the full automorphism group of a pseudoreal Riemann surface, and that any cyclic group of order $4n$, $n\geq 1$, 
is the full automorphism group of some pseudoreal Riemann surface.

 Motivated by the previous theorem, we study the possible extensions of a group $G$ by the cyclic group $C_2$.
The most general approach to this problem is through cohomology of finite groups (see \cite[Ch.~1]{Coh04}), but we will use easier tools.

Given an extension as in (\ref{ext}), consider an element $x\in\overline{G}\backslash G$. 
Since $G$ is normal in $\overline{G}$, this induces an automorphism $\phi_x$ of $G$ defined by conjugation by $x$ 
(from now on, we will denote by $\phi_p$ the conjugation by the element $p$).
Moreover  $g=x^2\in G$, $\phi_x^2=\phi_g$ and $g$ is fixed by $\phi_x$. Let $P(G)$ be the subset of ${\rm Aut}(G)\times G$ defined by $$P(G):=\{(\phi,g)\in {\rm Aut}(G)\times G:\ \phi^2=\phi_g,\ \  \phi(g)=g\}.$$
We can define an equivalence relation on $P(G)$ by $$(\phi,g)\sim (\phi\circ \phi_h, \phi(h)gh),\ \ \ \ \forall h\in G.$$
 Let $E(G)$ be the quotient set $P(G)\slash\sim$.

\begin{lemma} \label{lemma3.3.1} Given a group $G$, there exists a well defined function from the set of group extensions 
$$1\longrightarrow G\longrightarrow \overline{G}\longrightarrow C_2\longrightarrow 0,$$ to $E(G)$.
\end{lemma}

\proof For any such extension we can take an element $x\in \overline{G}\backslash G$ 
and construct the pair $(\phi_x,x^2)$. It is an easy exercise to check that different choices of $x$ lead to equivalent pairs. \qed

\begin{lemma} \label{3.3.2} Given an element $(\phi,g)\in P(G)$, there exist an extension $$1\longrightarrow G\longrightarrow \overline{G}\longrightarrow C_2\longrightarrow 0,$$ and an element $x\in \overline{G}\backslash G$ such that $\phi=\phi_x$ and $x^2=g$.
\end{lemma}

\begin{proof} Consider the group $\overline{G}$ defined by $$\overline{G}:=(G\rtimes_F \mathbb{Z})/\langle (g^{-1},z^2)\rangle,$$ where $\mathbb{Z}$ is the cyclic group generated by $z$ and $F$ is the homomorphism induced by $$F:\mathbb{Z}\longrightarrow {\rm Aut}(G),\ \ \ \ z\mapsto \phi.$$ The subgroup $\langle (g^{-1},z^2)\rangle$ is normal in $G\rtimes_F \mathbb{Z}$, so $\overline{G}$ is a group. Clearly $G$ injects into $\overline{G}$ through $a\mapsto (a,1)$, and we have that $$\overline{G}=\{(g,1)\ ,\ g\in G\}\cup \{(g,z)\ ,\ g\in G\},$$ because for $(p,z^m)\in G\rtimes_{F}\mathbb{Z}$ we have two cases $$[(p,z^{m})]=[(pg^{\frac{m}{2}},1)]\ \ \ \ \ \ {\rm for\ even}\ m,$$ $$[(p,z^{m})]=[(pg^{\frac{m-1}{2}},z)]\ \ \ \ \ {\rm for\ odd}\ m,$$
\noindent so $|\overline{G}|=2|G|$ and we have $\overline{G}\slash G\cong C_2$. Moreover $\phi_{(1,z)}(h,1)=(\phi(h),1),$ and $(1,z)^2=(g,1),$ so we can choose $x$ as $(1,z)$.
\end{proof}

\begin{theorem}\label{theorem} There is a bijection between the set of isomorphism classes of extensions of $G$ by $C_2$, and $E(G)$.
Moreover, one such extension is split if and only if it corresponds to a pair $[(\phi,g)]$ where $g=e$,
and it is a direct product if and only if one can choose $\phi={\rm id}_G, g=e$.
\end{theorem}

\begin{proof} Given an extension as in (\ref{ext}) 
we can associate to it the class $[(\phi_x,x^2)]$ by Lemma \ref{lemma3.3.1}. 
It is easy to see that an isomorphic extension leads to the same pair. 
Conversely, by Lemma \ref{3.3.2}, we can associate to every pair $(\phi,g)\in P(G)$ an extension of $G$ defined by $$A=\left(G\rtimes_F \mathbb{Z}\right)\slash \langle (g^{-1},x^2)\rangle$$ as we did above. Every pair $(\phi\circ \phi_h,\phi(h)gh)$ equivalent to $(\phi,g)$ will give us another group $$B=(G\rtimes_{F'}\mathbb{Z})\slash \langle((\phi(h)gh)^{-1},y^2)\rangle,$$ where $\mathbb{Z}=\langle y\rangle$, $h\in G$ and $F':\mathbb{Z}\longrightarrow {\rm Aut}(G)$ is induced by $y\mapsto \phi\circ \phi_h$. An isomorphism $\alpha:A\longrightarrow B$ is induced by $\alpha(g,1)=(g,1)$, $\alpha(1,x)=(\phi(h)^{-1},y)$. It is well defined because $$\alpha(g^{-1},x^2)=(g^{-1},1)(\phi(h)^{-1},y)(\phi(h)^{-1},y)=((\phi(h)gh)^{-1},y^2)$$ and clearly $\alpha|_G={\rm Id}_G$.

The exact sequence (\ref{ext}) is split if and only if $\overline{G}\backslash G$ has an order $2$ element $p$, which gives us the desired pair $(\phi_p,e)$.  
Moreover $\overline{G}\cong G\times C_2$ if and only if one can choose $\phi_p={\rm Id}_G$.
\end{proof}

\begin{corollary} Let $G$ be a finite group. Any extension of $G$ by $C_2$ is split if $Z(G)$ is trivial and one of the following holds
\begin{enumerate}
\item  ${\rm Out}(G):={\rm Aut}(G)\slash {\rm Inn}(G)$ has no involutions;
\item  ${\rm Inn}(G)$ has a group complement in ${\rm Aut}(G)$.
\end{enumerate}
\end{corollary}

\begin{proof} Given a group extension as in (\ref{ext}),
we can associate to it a pair $(\phi,g)$, such that $\phi^2=\phi_g$. 
The class $[\phi]$  in ${\rm Out}(G)$ satisfies $[\phi]^2=[1]$, but ${\rm Out}(G)$ has no order $2$ elements, so $\phi\in {\rm Inn}(G)$. 
In that case $(\phi,g)\sim (\phi\circ \phi^{-1},g')=({\rm Id}_G,g')$ for some $g'\in G$, where ${\rm Id}_G^2=\phi_{g'}$. Since $Z(G)=\{1\}$, then $g'=e$. 
Thus $(\phi,g)\sim ({\rm Id}_G,e)$ so by  Theorem \ref{theorem} every extension of $G$ by $C_2$ will be a direct product $G\times C_2$. This proves (i).

Let $H$ be the group complement of ${\rm Inn}(G)$ in ${\rm Aut}(G)$, that is $${\rm Aut}(G)= H\cdot {\rm Inn}(G),\ \ \ \ H\cap {\rm Inn}(G)=\{1\}.$$ 
If $(\phi,g)\in P(G)$, then $\phi^2=\phi_g$ and $\phi(g)=g$. We have that ${\rm Aut}(G)=H\cdot {\rm Inn}(G)$, so $\phi\in {\rm Aut}(G)$ can be written as $\phi=\varphi\circ \phi_h$ with $\varphi\in H$ and $\phi_h\in {\rm Inn}(G)$, so $(\phi,g)\sim (\varphi\circ \phi_h,g)\sim (\varphi,g')$ for some $g'\in G$. We also have $\varphi^2\in H\cap {\rm Inn}(G)=\{1\}$ so $\varphi^2=1$, but $\varphi^2=\phi_{g'}$ so $\phi_{g'}=1$, which is equivalent to $g'=e$ because $Z(G)=\{1\}$. In that case $(\varphi,g')=(\varphi,e)$ so we get the desired equality $[(\phi,g)]=[(\varphi,e)]$ and we obtain (ii) by Theorem \ref{theorem}.
\end{proof}

 If we translate the previous result for pseudoreal Riemann surfaces, we get the following. The 
 second condition had already been proved in \cite[Corollary 4.3]{DebEms99} with different techniques.

\begin{corollary}\label{comp}
Let $G$ be the conformal automorphism group of a Riemann surface $X$. 
Suppose that $Z(G)=\{1\}$ and that either ${\rm Out}(G)$ has no involutions or ${\rm Inn}(G)$ has group complement 
in ${\rm Aut}(G)$. Then $X$ is not pseudoreal.
\end{corollary}

\subsection{Maximal groups}

A. Hurwitz proved that the order of the conformal automorphism group of a Riemann surface of genus $g\ge 2$ is bounded above by $84(g-1)$ (see \cite[p.~424]{Hur93}), and there are infinitely many Riemann surfaces whose conformal automorphism group attains that bound. 
The first example of such groups was the order $168$ group ${\rm PSL}(2,7)$, which is the conformal automorphism group of the Klein's quartic. 
In the case of pseudoreal Riemann surfaces, the Hurwitz bound is never attained because all such surfaces have conformal automorphism groups of signature $(0;[2,3,7])$, which is an odd signature. 
This also follows from \cite[Theorem 5.4]{Wolfart} since such Riemann surfaces are quasiplatonic, 
i.e. the genus of $X/{\rm Aut}(X)$ is zero and the quotient 
$X\to X/{\rm Aut}(X)$ has at most three branch points.
For pseudoreal Riemann surfaces there is a better upper bound, as we see in the following

\begin{theorem}\label{maximalll} \cite[Theorem 5.1]{BuConCo10} If $X=\mathbb{H}\slash\Gamma$ is a pseudoreal Riemann surface of genus $g$ with full automorphism group $G$, then $|G|\le 12(g-1)$. Moreover, if $|G|=12(g-1)$ and $G=\Delta\slash\Gamma$ then the signature of $\Delta$ is $(1;-;[2,3])$.
\end{theorem}

If a pseudoreal Riemann surface $X$ has genus $g$ and full automorphism group of order $12(g-1)$, 
we will say that $X$ has \emph{maximal full group}. 
Using the programs in Section \ref{prog}  
we found that the minimum genus for which there exists a pseudoreal Riemann surface with maximal full group is $g=14$, with conformal automorphism group ${\rm ID}(78,1)$ and full automorphism group ${\rm ID}(156,7)$.
In  \cite[Theorem 5.5]{BuConCo10} the authors proved that there exists pseudoreal Riemann surfaces with maximal 
full group for infinitely many genera. The groups that they got are non abelian, this inspired us to prove the next result.

\begin{theorem} \label{abeliann} If a pseudoreal Riemann surface $X$ has maximal full group, 
then its conformal automorphism group is not abelian and $X/Z({\rm Aut}(X))$ has positive genus.
\end{theorem}

\begin{proof} Suppose that the conformal automorphism group $G$ of $X$ is abelian. 
First observe that $G$ has order $6(g-1)$ and the Fuchsian signature associated to $G$ is $(0;[2,2,3,3])$ (Theorem \ref{maximalll} and Theorem \ref{5.1}). 
By \cite[Theorem 7.1]{BuCiCo02} $G$ is cyclic since otherwise it would be a quotient of $C_2\oplus C_3\cong C_6$ (and thus cyclic).
By Table \ref{GEN2} we know that there is no conformal automorphism group of a pseudoreal Riemann surface of order $6$ in genus $2$, so we can assume $g>2$. However, in this case  $6(g-1)>2g+2$, and a generator of $G$ will be an element of order $>2g+2$. By \cite[Corollary 1]{Sin74b} $X$ is not pseudoreal, contradicting the hypothesis. This proves the first half of the statement.

The second statement follows from the first one and Corollary \ref{cencor}.
\end{proof}

\section{Classification}\label{abc}
 In this section we will recall what are conformal and full automorphism groups of pseudoreal 
 Riemann surfaces up to genus $4$ and we will extend such classification up to genus $10$. \\

\noindent \textbf{Genus 2.} By  \cite[Theorem 2]{CarQuer02} we know that if $X$ is a curve of genus $2$ 
defined over a field of characteristic not equal to $2$, and ${\rm Aut}(X)\not\cong C_2$, 
then $X$ can be defined over its field of moduli. 
In particular, if $X$ is a pseudoreal Riemann surface of genus $2$, then ${\rm Aut}(X)\cong C_2$. 
The latter result was also obtained in \cite[Theorem 4.1]{BuConCo10} via NEC groups and epimorphisms, 
obtaining $C_4$ as the only possible full automorphism group in genus $2$ (see Table \ref{GEN2}).
An algebraic model for a pseudoreal curve of genus $2$ is Earle's example  \cite{E71}
$$X\ :\ y^2=x(x^2-a^2)(x^2+ta^2x-a),$$ where $a=e^{\frac{2\pi i}{3}}$ and $t\in\mathbb{R}^{+}-\{1\}$. \\
 
\noindent \textbf{Genus 3.} 
A hyperelliptic curve of genus three can be defined over its field of moduli unless its automorphism group is isomorphic 
to $C_2\times C_2$ (see \cite[\S 4.5]{LR11}).

In the non-hyperelliptic case, in \cite[Theorem 0.2]{ArtQui12} the authors proved that if $X$ 
is a smooth plane quartic such that ${\rm |Aut}(X)|>4$, then $X$ can be defined over its field of moduli, 
since all the other possible automorphism groups have odd signature.
Moreover, the authors proved that a pseudoreal plane quartic  has automorphism group $C_2$.
The same result was obtained in \cite[Proposition 3.5]{BuCo14} via NEC groups and epimorphisms, obtaining $C_4$ and $C_4\times C_2$ as the only possible full automorphism groups in genus $3$ (see Table \ref{GEN3}). 
 
 The equation of a pseudoreal hyperelliptic curve of genus $3$ is given in \cite[p.~82]{Hugg05}  
 $$X\ :\ y^2=(x^2-a_1)\left(x^2+\frac{1}{\overline{a_1}}\right)(x^2-a_2)\left(x^2+\frac{1}{\overline{a_2}}\right),$$
 where the coefficients must satisfy certain special conditions.
  An explicit pseudoreal plane quartic is  given by 
  $$X\ :\ y^4+y^2(x-a_1z)\left(x+\frac{1}{a_1}z\right)+(x-a_2z)\left(x+\frac{1}{\overline{a_2}}z\right)(x-a_3z)\left(x+\frac{1}{\overline{a_3}}z\right)=0,$$ where $a_1=1, a_2=1-i, a_3=2(i-1)$  \cite[Proposition 4.3]{ArtQui12}. \\
 
\item \textbf{Genus 4.} In \cite[Theorem 4.3]{BuCo14} the authors find that the only possible full 
automorphism groups for pseudoreal Riemann surfaces are $C_4,C_8$ and the Frobenius group $F20$ (see Table \ref{GEN4}). 
An algebraic model for a pseudoreal curve of genus $4$ when ${\rm Aut}^{+}(X)$ is $C_2$ with Fuchsian signature $[0;2^{10}]$ 
has been given by Shimura  \cite{Shi72} 
$$y^2=x^5+(a_1x^6-\overline{a_1}x^4)+(a_2x^7+\overline{a_2}x^3)+(a_3x^8-\overline{a_3}x^2)+(a_4x^9+\overline{a_4}x)+(x^{10}-1),$$ 
and has full group $C_4$, where the coefficients $a_i$ and $\overline{a_j}$ are algebraically independent over $\mathbb{Q}$. When ${\rm Aut}^{+}(X)$ is $C_4$, we have a hyperelliptic example in \cite[p.~82]{Hugg05} given by $$y^2=x(x^4-b_i)\left(x^4+\frac{1}{\overline{b_i}}\right),$$ which has full group $C_8$.

\begin{example}  
We now provide the equations of a non-hyperelliptic pseudoreal curve of genus $4$  
with automorphism group isomorphic to $D_5$. This example was kindly communicated to us by Jeroen Sijsling.
Let $X$ be the complete intersection of the quadric and the cubic in $\pp^3$ defined as the zero sets of the following polynomials:
\[
F_2 := (-3i + 2)(x_1^2+x_2^2+x_3^2+x_4^2) + (9i - 2)(x_1x_2+x_2x_3+x_3x_4) - 6i(x_1x_3 +x_1x_4  + x_2x_4),
\]
\[
F_3 := (-2i + 1)(x_1^2x_2 +x_2^2x_3+x_3^2x_4-x_1x_2^2-x_2x_3^2-x_3x_4^2)+ \]
\[+ 4i(x_1^2x_3 - x_1^2x_4  - x_1x_3^2 + x_1x_4^2  + x_2^2x_4  - x_2x_4^2).
\]
Observe that, since $X$ is canonically embedded, all elements of ${\rm Aut}(X)$ are induced by 
projectivities in ${\rm PGL}(4,\cc)$.
The curve $X$ contains the following elements in its automorphism group
\[
T_1=\left(
\begin{array}{cccc}
-1 &  1&  0&  0\\
-1 &  0 &  1&  0\\
-1 &  0 &  0 &  1\\
-1 &  0 &  0 &  0
\end{array}
\right),\quad 
T_2 = \left(
\begin{array}{cccc}
-1 &  0 &  0 &  0\\
-1 &  0 &  0 & 1\\
-1 &  0 &  1 &  0\\
-1 &  1 &  0 &  0
\end{array}
\right),
\]
which generate a subgroup $G$ of ${\rm Aut}(X)$ isomorphic to the dihedral group $D_5$:
\[
G\to D_5=\langle a,b|a^2,b^5,aba^{-1}b\rangle,\quad T_1\mapsto b,\ T_2\mapsto a.
\]
Actually, $G={\rm Aut}(X)$ and can be proved as follows.
By \cite[Proposition 2.5, II, 1]{AKuIKu90} and \cite[Table 4]{MSSV}, if $G$ were properly contained in  ${\rm Aut}(X)$, then ${\rm Aut}(X)$ would be isomorphic 
to the symmetric group $S_5$ and $G$ to the unique subgroup of $S_5$ isomorphic to $D_5$ 
(which can be generated by $(13)(45)$ and $(12345)$).
Moreover, the normalizer of $G$ in ${\rm Aut}(X)$ would be a group of 
order $20$ containing an element $S$  (corresponding to $(1534)$) which acts by conjugacy on $G$ as $a\mapsto a,\ b\mapsto b^2$.

Observe that the following matrix  acts by conjugacy on $T_1,T_2$ in the same way as  $S$:
\[
M = \left(
\begin{array}{cccc}
1 &  0 &  0&  0\\
1 &  0 & -1 &  1\\
0 &  1 & -1 &  1\\
0 &  1 & -1 &  0.
\end{array}
\right).
\]
This implies that $A=M^{-1}S$ belongs to the centralizer of $G$ in ${\rm PGL}(4,\cc)$.
Since the trace of $T_1$ is not zero and since $T_2$ has order two, the $AT_1=T_1A$ and $AT_2=\pm T_2A$ in ${\rm GL}(4,\cc)$.
An explicit computation with the help of Magma \cite{Magma} allows to identify all such matrices $A$ (which turn out to depend on one parameter) 
and to check that $MA$ never gives an automorphism of $X$.

Observe that the matrix $M$ gives an isomorphism between $X$ and its conjugate $\bar X$.
Finally, one can verify that the product of $M$ with any element of $G$ is not an involution,
showing that $X$ has no anticonformal involutions.
\end{example}

As far as the authors know, there is no explicit example of a pseudoreal Riemann surface of genus $4$ 
with conformal automorphism group $C_2$ with signature $(2;[2,2])$.

\begin{theorem} \label{class} Two finite groups $G$ and $\overline{G}$ are the conformal and full automorphism groups of a pseudoreal Riemann surface $X$ of genus $5\le g\le 10$ if and only if $G={\rm Aut}^{+}(X)$ and $\overline{G}={\rm Aut}^{\pm}(X)$ appear in the corresponding table by genus among Tables \ref{table5}, \ref{table6}, \ref{table7}, \ref{table8}, \ref{table9} and \ref{table10}.
\end{theorem}

\proof
The classification is obtained through a case by case analysis, carried out by means
of the Magma program which is described in the next section.  
The main steps of the algorithm are the following.
\begin{enumerate}
\item Fixed a genus $5\le g\le 10$ we consider the complete list $L$ of triples $(G,s,v)$, where $G$ is a group acting on 
Riemann surfaces $X$ of genus $g$, $s$ its signature and $v$ the generating vector of the action (see \S 2.2). 
These data are given us by the Magma program of J. Paulhus (see \cite{Pau15}). 
\item From the list $L$ we select only the triples where $G$ has even order and $s$ is an even signature (see  Theorem \ref{5.1}).
\item  We separate the finitely maximal and the non finitely maximal signatures. 
 \item In  case the signature is finitely maximal, 
 we  first consider all the non-split extensions $\overline{G}$ of $G$ by $C_2$ (Theorem \ref{bagofull}). 
 For any such extension we check the existence of an epimorphism $\theta:\Delta\to\overline G$, where $\Delta$ is a NEC group  
whose signature is obtained from $s$ by Theorem \ref{5.1}, and $\ker(\theta)$ is a torsion free Fuchsian group.
If there is one such epimorphism, $G$ and $\bar G$ are the conformal and full automorphism groups respectively of a pseudoreal Riemann surface.
 \item In case the signature is non finitely maximal, we apply Lemmas \cite[Lemma 4.1, Lemma 4.2, Lemma 4.3]{BaGo10} and Lemma \ref{BG4} to identify pseudoreal automorphism groups.  \qed
\end{enumerate}

Observe that for any case in the tables of section \ref{app} there exists a pseudoreal curve with those properties.
However, finding an explicit algebraic model for such curves is a difficult problem.
In \cite{Shi72} G. Shimura  provided hyperelliptic examples for any even genus $g\geq 2$ with automorphism group of order two.
Later the hyperelliptic case has been completely characterized in the following theorem by B. Huggins, which 
completes previous work in \cite{BuTur02}.

\begin{theorem} \cite[Theorem 5.0.5]{Hugg05} Let $X$ be a hyperelliptic curve defined over $\mathbb{C}$ such that $M_{\mathbb{C}\slash\mathbb{R}}(X)=\mathbb{R}$. Then $X$ is pseudoreal  if and only if it is isomorphic to either $y^2=f(x)$, or $y^2=g(x)$, where 
$$f(x)=\prod\limits_{i=1}^r(x^n-a_i)\left(x^n+\frac{1}{\overline{a_i}}\right),\quad g(x)=x\prod\limits_{i=1}^s(x^m-b_i)\left(x^m+\frac{1}{\overline{b_i}}\right),$$ 
$m,n,r,s$ are non negative integers such that $2nr>5$, $sm$ is even, $r$ is odd if $n$ is odd, and $a_i, b_i$ satisfy additional conditions that can be found in \cite[Pag.~82]{Hugg05}. Moreover, these curves have automorphism groups isomorphic to $C_2\times C_n$ and $C_{2n}$, respectively.
\end{theorem}

In \cite[Theorem 2.]{E71} the author gives an example of a pseudoreal Riemann surface $X$ of genus $5$ 
with an order $4$ anticonformal element called $f$, which generates ${\rm Aut}^{\pm}(X)\cong C_4$ 
(see Figure \ref{fig:awesome_image}). There are exactly $2$ possible conformal actions 
of $C_2$ on pseudoreal Riemann surfaces of genus $5$, having signatures $(3;[-])$ and $(1;[2^8])$. 
Since $f^2$ has no fixed points, the conformal 
action is $C_2$ with signature $(3;[-])$ (see Table \ref{table5}).

\begin{figure}[h!]
\centering
    \includegraphics[width=0.3\textwidth]{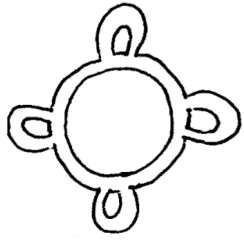}
    \caption{Earle's picture of his genus $5$ example} \label{fig:awesome_image}
\end{figure}

More non-hyperelliptic examples of pseudoreal curves have been introduced by B. Huggins in \cite{Hugg05} 
and later by A. Kontogeorgis for $p$-gonal curves \cite{Kon09}.
In \cite{Hid09} R. Hidalgo found a nice example of a non-hyperelliptic pseudoreal curve of genus $17$ with automorphism group 
$C_2^5$ which is a covering of the genus two example given by Earle.
Recently in \cite{ACHQ16} the authors constructed a tower of Riemann surfaces, going from Earle's example in genus two to 
Hidalgo's example in genus $17$, which contains non-hyperelliptic examples of genus $5$ and $9$. 
 
Finally, we apply Theorem \ref{class} to the classification of pseudoreal plane quintics.

\begin{corollary} \label{quinticaaa} Two finite groups $G$ and $\overline{G}$ are 
the conformal and full automorphism groups of a pseudoreal plane quintic $X$ if and only 
if $G={\rm Aut}^{+}(X)$ and $\overline{G}={\rm Aut}^{\pm}(X)$ are in a row of Table \ref{pseudoquintic}.

\begin{table}[!h]
\centering\renewcommand{\arraystretch}{1.3}\setlength{\tabcolsep}{2pt}
\medskip
    \begin{tabular}{| c | c | c | c | p{3.5cm} |}
    \hline
    $\rm Aut^{+}(X)$ & Fuchsian signature & $\rm Aut^{\pm}(X)$ & NEC signature & Generating Vector \\ \hline    
    $C_4$ & $(0;4^6)$ & $C_8$ & $(1;-;[4^3])$ & $(a;a^2,a^2,a^2)$ \\ \hline    
    $C_2$ & $(2;[2^6])$ & $C_4$ & $(3;-;[2^3])$ & $(a,a,a;a^2,a^2,a^2)$ \\ \hline 
     \end{tabular}
      \vspace{0.5cm}
     \caption{Possible automorphism groups for pseudoreal plane quintics}
\label{pseudoquintic}
\end{table}
\end{corollary}
\vspace{-0.3cm}

\begin{proof} 
In \cite{BaBa16} the authors classified the automorphism groups of plane quintics defined 
over an algebraically closed field of characteristic zero, giving a smooth plane model 
for every group.
This classification, together with Table \ref{table6}, gives a list of the possible conformal and 
full groups of pseudoreal plane quintics. The possible conformal automorphisms groups 
are $D_5, C_2, C_4$:
\begin{enumerate}
\item ${\rm Aut}(X)\cong D_{5}$, with group generators $\sigma_1([x:y:z])= [x:\epsilon_5y:\epsilon_5^2z]$ and $\sigma_2([x:y:z])= [z:y:x]$, 
with a smooth plane model $$x^5+y^5+z^5+ax^2yz^2+bxy^3z=0,$$ 
where $a,b\in\cc$ with $(a,b)\neq (0,0)$. In this case the covering $X\longrightarrow X\slash {\rm Aut}(X)$ has signature $(0;[2^6])$.
\item ${\rm Aut}(X)\cong C_4$, with generator $\sigma([x:y:z])=[x:y:\epsilon_4z]$, with smooth plane model 
$$z^4L_{1,z}+L_{5,z}=0,$$ 
where $L_{i,z}\in \cc[x,y]$ are homogeneous polynomials with $\deg(L_{i,z})=i$.
In this case the covering $X\longrightarrow X\slash {\rm Aut}(X)$ has signature $(0;[4^6])$.
\item ${\rm Aut}(X)\cong C_2$, with generator $\sigma([x:y:z])=[x:y:-z]$, with smooth plane model 
$$z^4L_{1,z}+z^2L_{3,z}+L_{5,z}=0,$$ 
where $L_{i,z}\in \cc[x,y]$ are homogeneous polynomials with $\deg(L_{i,z})=i$.
In this case the covering $X\longrightarrow X\slash {\rm Aut}(X)$ has signature $(2;[2^6])$.
\end{enumerate}

We now prove that case (i) cannot occur. 
Suppose we have an isomorphism $f$ between $X$ and $\bar X$.
Then $f$ must have a representation as a $3\times 3$ matrix (see \cite{BaBa16}) 
and it must send the fixed points of the cyclic subgroup of order $5$ of ${\rm Aut}(X)$ to the same points 
for $\bar X$.  Such cyclic subgroup is generated by $\sigma_1$ for both curves and its fixed points are the fundamental points, 
so the matrix representing $f$ must be the composition of a permutation matrix with a diagonal matrix.  
Since $[x:y:z]\mapsto [z:y:x]$ is the only non-trivial permutation allowed, then a matrix representation 
for $f$, up to composing with an automorphism of $X$, is the following
$$\begin{bmatrix}
       1 & 0 & 0           \\[0.3em]
       0 & \epsilon_5^m & 0 \\[0.3em]
       0 & 0 & \epsilon_5^n
     \end{bmatrix},
$$  
where $\epsilon_5$ is a primitive fifth root of unity and $m,n$ are integers.
 Since $(J\circ f)^2={\rm id}$, the curve $X$ admits an anticonformal involution, thus it is not pseudoreal.
\end{proof}

\section{A Magma program} \label{prog}
 In this section we will present the Magma \cite{Magma}  program that we used to carry out the classification of full 
automorphism groups of pseudoreal Riemann surfaces done in the previous section.
Our program relies on Jennifer Paulhus' program {\bf GenVectMagmaToGenus20}, which is available at
\begin{center}
 \url{http://www.math.grinnell.edu/~paulhusj/monodromy.html} 
 \end{center}
 and is based on the paper \cite{Pau15}.
To run our program one first needs to download the packages 
{\bf genvectors.m}, {\bf searchroutines.m}, {\bf GenVectMagmaToGenus20} 
and save all of them in the same folder.
To access the data in Paulhus' program, for example for genus $4$, one has to write in Magma 

\begin{verbatim}
load "genvectors.m";
load "searchroutines.m";
L:=ReadData("Fullg20/grpmono04"€, test);
\end{verbatim}
 
where  ``test'' is a function taking as input a permutation group,
a signature (as a vector) and a generating vector (as a vector whose entries are permutations). 
For example when using the following function, the program gives the list of all triples $(G,s,v)$, where $G$ is a group of order bigger than $7$ acting 
on a Riemann surface of genus $4$ with signature $s$ and generating vector $v$.

%\begin{minted}[fontsize=\small,frame=single]{mupad}
\begin{verbatim}
test:=function(G,s,Lmonod)
 return  Order(G) gt 7;
end function;
\end{verbatim}

Thus this program allows to analyse the automorphisms groups of all Riemann surfaces up to genus $20$,
looking for certain properties specified by the function ``test''.\footnote{The data was computed using Magma V2.19-9. In newer versions of Magma an error may be returned for some genera (at least $5$ and $9$). 
See the warning in J. Paulhus' webpage.}
Observe that $G$ is not necessarily the complete automorphism group of some Riemann surface of the chosen genus (this will be one of the main issues in our program).

Given a genus $2\leq g\leq 20$ our program describes the automorphism group of all 
the pseudoreal Riemann surfaces of genus $g$. More precisely it gives the full automorphism group, 
the conformal automorphism group and its Fuchsian signature.
For each entry of the output there exists a pseudoreal Riemann surface of genus $g$ 
with such properties.
To run the program one needs to download the file \textbf{pseudoreal.m}, which is available here 
\begin{center}
\url{https://www.dropbox.com/s/k786b7a2vrmt22i/pseudoreal.m?dl=0}
\end{center}
and save it in the same folder as Paulhus' programs.
The program (again in the case of genus $4$) consists of the following lines  

%\begin{minted}[fontsize=\small,frame=single]{mupad}
\begin{verbatim}
load "genvectors.m"; 
load "searchroutines.m";
load "pseudoreal.m";
L:=ReadData("Fullg20/grpmono04", testpr);
PR(L);
\end{verbatim}

The output is a list whose entries are of the form $\langle \langle \,,\,\rangle, \langle \,,\,\rangle, [\dots]\rangle$,
where the first bracket contains the ID number of the full automorphism group, the second bracket contains the ID number 
of the conformal automorphism group and the final sequence is the corresponding Fuchsian signature (the first entry is the genus of the quotient 
by the conformal automorphism group).
A description of each of the functions contained in pseudoreal.m can be found in \cite{Rey16}.


\begin{bibdiv}
\begin{biblist}

 


\bib{ACHQ16}{article}{
    AUTHOR = {Artebani, M.},
    AUTHOR = {Carvacho, M.},
    AUTHOR = {Hidalgo, R. A.},
    AUTHOR = {Quispe, S.},
     TITLE = {Tower of Riemann Surfaces which cannot be defined over their field of Moduli},
   JOURNAL = {Glasgow Mathematical Journal},
   % VOLUME = {FirstView},
      YEAR = {2016},
     PAGES = {1--15},
       }
       
\bib{ArtQui12}{article}{
    AUTHOR = {Artebani, M.},
    AUTHOR = {Quispe, S.},
     TITLE = {Fields of moduli and fields of definition of odd signature
              curves},
   JOURNAL = {Arch. Math. (Basel)},
  FJOURNAL = {Archiv der Mathematik},
    VOLUME = {99},
      YEAR = {2012},
    NUMBER = {4},
     PAGES = {333--344},
      ISSN = {0003-889X},
     CODEN = {ACVMAL},
   MRCLASS = {14H37},
  MRNUMBER = {2990152},
MRREVIEWER = {Francesco Polizzi},
     %  DOI = {10.1007/s00013-012-0427-6},
       URL = {http://dx.doi.org/10.1007/s00013-012-0427-6},
       }

\bib{BaGo10}{article}{
    AUTHOR = {Baginski, C.},
    AUTHOR = {Gromadzki, G.},
     TITLE = {Minimal genus problem for pseudo-real Riemann surfaces},
   JOURNAL = {Arch. Math. (Basel)},
  FJOURNAL = {Archiv der Mathematik},
    VOLUME = {95},
      YEAR = {2010},
    NUMBER = {5},
     PAGES = {481--492},
       }
       
             
       \bib{Bars05}{article}{
    AUTHOR = {Bars, F.},
     TITLE = {Automorphism groups of genus 3 curves},
   JOURNAL = {Number Theory Seminar UAB-UB-UPC on Genus 3 curves. Barcelona},
      YEAR = {2005},
       }
       
 

\bib{Magma}{article}{
    AUTHOR = {Bosma, W.},
    AUTHOR = {Cannon, J.},
    AUTHOR = {Playoust, C.},
     TITLE = {The {M}agma algebra system. {I}. {T}he user language},
      NOTE = {Computational algebra and number theory (London, 1993)},
   JOURNAL = {J. Symbolic Comput.},
    VOLUME = {24},
      YEAR = {1997},
    NUMBER = {3-4},
     PAGES = {235--265}
}

 
\bib{BaBa16}{article}{
    AUTHOR = {Badr, E.},
    AUTHOR = {Bars, F.},
     TITLE = {Automorphism {G}roups of {N}onsingular {P}lane {C}urves of {D}egree 5},
   JOURNAL = {Comm. Algebra},
  FJOURNAL = {Communications in Algebra},
    VOLUME = {44},
      YEAR = {2016},
    NUMBER = {10},
     PAGES = {4327--4340},
      ISSN = {0092-7872},
   MRCLASS = {14H37 (14H45 14H50)},
  MRNUMBER = {3508302},
 %      DOI = {10.1080/00927872.2015.1087547},
       URL = {http://dx.doi.org/10.1080/00927872.2015.1087547},
}
       
    \bib{Buj82}{article}{
    AUTHOR = {Bujalance, E.},
     TITLE = {Normal N.E.C. signatures},
   JOURNAL = {Illinois J. Math.},
    VOLUME = {26},
      YEAR = {1982},
     PAGES = {519--530},
       }   

\bib{BuCiCo02}{article}{
    AUTHOR = {Bujalance, E.},
    AUTHOR = {Cirre, F. J.},
    AUTHOR = {Conder, M.},
     TITLE = {On extendability of group actions on compact {R}iemann
              surfaces},
   JOURNAL = {Trans. Amer. Math. Soc.},
  FJOURNAL = {Transactions of the American Mathematical Society},
    VOLUME = {355},
      YEAR = {2003},
    NUMBER = {4},
     PAGES = {1537--1557 (electronic)},
      ISSN = {0002-9947},
     CODEN = {TAMTAM},
   MRCLASS = {20H10 (14H55 30F10)},
  MRNUMBER = {1946404},
MRREVIEWER = {Grzegorz Gromadzki},
     %  DOI = {10.1090/S0002-9947-02-03184-7},
       URL = {http://dx.doi.org/10.1090/S0002-9947-02-03184-7},
       }

\bib{BS}{book}{
    AUTHOR = {Bujalance, E.},
    AUTHOR = {Cirre, F. J.},
    AUTHOR = {Gamboa, J. M.},
    AUTHOR = {Gromadzki, G.},
     TITLE = {Symmetries of compact Riemann surfaces},
    SERIES = {Lecture Notes in Mathematics},
    VOLUME = {2007},
 PUBLISHER = {Springer-Verlag},
   ADDRESS = {Berlin},
      YEAR = {2010},
}
              
 

\bib{BuCo14}{article}{
    AUTHOR = {Bujalance, E.},
    AUTHOR = {Costa, A. F.},
     TITLE = {Automorphism groups of pseudo-real Riemann surfaces of low genus},
   JOURNAL = {Acta Math. Sin (Engl. Ser.)},
    VOLUME = {30},
      YEAR = {2014},
    NUMBER = {1},
     PAGES = {11--22},
       }
       
        
\bib{BuConCo10}{article}{
    AUTHOR = {Bujalance, E.},
    AUTHOR = {Conder, M.},
    AUTHOR = {Costa, A. F.},
     TITLE = {Pseudo-real {R}iemann surfaces and chiral regular maps},
   JOURNAL = {Trans. Amer. Math. Soc.},
  FJOURNAL = {Transactions of the American Mathematical Society},
    VOLUME = {362},
      YEAR = {2010},
    NUMBER = {7},
     PAGES = {3365--3376},
      ISSN = {0002-9947},
     CODEN = {TAMTAM},
   MRCLASS = {30F10 (14H55 30F50 57M15)},
  MRNUMBER = {2601593},
MRREVIEWER = {Jose Manuel Gamboa},
    %   DOI = {10.1090/S0002-9947-10-05102-0},
       URL = {http://dx.doi.org/10.1090/S0002-9947-10-05102-0},
       }

\bib{BK}{book}{
    AUTHOR = {Bujalance, E.},
    AUTHOR = {Etayo, J.J.},
    AUTHOR = {Gamboa, J. M.},
    AUTHOR = {Gromadzki, G.},
     TITLE = {Automorphism groups of compact bordered Klein surfaces. A combinatorial approach},
    SERIES = {Lecture Notes in Mathematics},
    VOLUME = {1439},
 PUBLISHER = {Springer},
   ADDRESS = {Berlin},
      YEAR = {1990},
}
        
 
\bib{BuTur02}{article}{
    AUTHOR = {Bujalance, E.},
    AUTHOR = {Turbek, P.},
     TITLE = {Asymmetric and pseudo-symmetric hyperelliptic surfaces},
   JOURNAL = {Manuscripta Math.},
  FJOURNAL = {Manuscripta Mathematica},
    VOLUME = {108},
      YEAR = {2002},
    NUMBER = {1},
     PAGES = {1--11},
       }
       
       
 

\bib{CarQuer02}{article}{
AUTHOR = {Cardona, G.},
AUTHOR = {Quer, J.},
     TITLE = {Field of moduli and field of definition for curves of genus 2},
 BOOKTITLE = {Computational aspects of algebraic curves},
    SERIES = {Lecture Notes Ser. Comput.},
    VOLUME = {13},
     PAGES = {71--83},
 PUBLISHER = {World Sci. Publ., Hackensack, NJ},
      YEAR = {2005},
   MRCLASS = {14H10},
  MRNUMBER = {2181874},
MRREVIEWER = {Tanush Shaska},
       URL = {http://dx.doi.org/10.1142/9789812701640_0006},
}


\bib{Coh04}{book}{
    AUTHOR = {Adem, A.},
    AUTHOR = {Milgram, J. R.},
     TITLE = {Cohomology of Finite Groups},
    SERIES = {Grundlehren der mathematischen Wissenschaften},
    VOLUME = {309},
 PUBLISHER = {Springer},
   ADDRESS = {Berlin},
      YEAR = {2004},
}

\bib{DebEms99}{article}{
AUTHOR = {D{\`e}bes, P.}
AUTHOR = {Emsalem, M.}
     TITLE = {On fields of moduli of curves},
   JOURNAL = {J. Algebra},
  FJOURNAL = {Journal of Algebra},
    VOLUME = {211},
      YEAR = {1999},
    NUMBER = {1},
     PAGES = {42--56},
      ISSN = {0021-8693},
   MRCLASS = {14H25 (14H30)},
  MRNUMBER = {1656571},
MRREVIEWER = {Helmut V{\"o}lklein},
  %     DOI = {10.1006/jabr.1998.7586},
       URL = {http://dx.doi.org/10.1006/jabr.1998.7586},
       }
       
\bib{E71}{article}{
AUTHOR = {Earle, C. J.},
     TITLE = {On the moduli of closed {R}iemann surfaces with symmetries},
 BOOKTITLE = {Advances in the theory of riemann surfaces ({P}roc. {C}onf., {S}tony {B}rook, {N}.{Y}., 1969)},
     PAGES = {119--130. Ann. of Math. Studies, No. 66},
 PUBLISHER = {Princeton Univ. Press, Princeton, N.J.},
      YEAR = {1971},
   MRCLASS = {30A48 (58D15)},
  MRNUMBER = {0296282},
MRREVIEWER = {L. Keen},
}

\bib{EstIzq06}{article}{
    AUTHOR = {Est{\'e}vez, J. L.},
    AUTHOR = {Izquierdo, M.},
     TITLE = {Non-normal pairs of non-{E}uclidean crystallographic groups},
   JOURNAL = {Bull. London Math. Soc.},
  FJOURNAL = {The Bulletin of the London Mathematical Society},
    VOLUME = {38},
      YEAR = {2006},
    NUMBER = {1},
     PAGES = {113--123},
      ISSN = {0024-6093},
     CODEN = {LMSBBT},
   MRCLASS = {20H15 (20H10 30F10)},
  MRNUMBER = {2201610},
MRREVIEWER = {Ernesto Mart{\'{\i}}nez},
  %     DOI = {10.1112/S0024609305017984},
       URL = {http://dx.doi.org/10.1112/S0024609305017984},
       }
    
 
       
\bib{Gre63}{article}{
    AUTHOR = {Greenberg, L.},
     TITLE = {Maximal {F}uchsian groups},
   JOURNAL = {Bull. Amer. Math. Soc.},
  FJOURNAL = {Bulletin of the American Mathematical Society},
    VOLUME = {69},
      YEAR = {1963},
     PAGES = {569--573},
      ISSN = {0002-9904},
   MRCLASS = {10.21 (20.65)},
  MRNUMBER = {0148620},
MRREVIEWER = {J. R. Smart},
}


 
       
\bib{HamHerr03}{article}{
AUTHOR = {Hammer, H.},
AUTHOR=  {Herrlich, F.},
     TITLE = {A remark on the moduli field of a curve},
   JOURNAL = {Arch. Math. (Basel)},
  FJOURNAL = {Archiv der Mathematik},
    VOLUME = {81},
      YEAR = {2003},
    NUMBER = {1},
     PAGES = {5--10},
      ISSN = {0003-889X},
     CODEN = {ACVMAL},
   MRCLASS = {14H25 (11G30)},
  MRNUMBER = {2002711},
    %   DOI = {10.1007/s00013-003-4649-5},
       URL = {http://dx.doi.org/10.1007/s00013-003-4649-5},
       }
 

\bib{Hid09}{article}{
    AUTHOR = {Hidalgo, R. A.},
     TITLE = {Non-hyperelliptic Riemann surfaces with real field of moduli but not definable over the reals},
   JOURNAL = {Arch. Math.},
    VOLUME = {93},
      YEAR = {2009},
     PAGES = {219--222},
       }
 
       
\bib{HidQui15}{article}{
    AUTHOR = {Hidalgo, R. A.},
    AUTHOR = {Quispe, S.},
     TITLE = {Fields of moduli of some special curves},
   JOURNAL = {J. Pure Appl. Algebra},
  FJOURNAL = {Journal of Pure and Applied Algebra},
    VOLUME = {220},
      YEAR = {2016},
    NUMBER = {1},
     PAGES = {55--60},
      ISSN = {0022-4049},
   MRCLASS = {14H45 (14H37)},
  MRNUMBER = {3393450},
  %     DOI = {10.1016/j.jpaa.2015.05.042},
       URL = {http://dx.doi.org/10.1016/j.jpaa.2015.05.042},
       }
       
 

\bib{Hugg04}{article}{
    AUTHOR = {Huggins, B.},
     TITLE = {Fields of moduli of hyperelliptic curves},
   JOURNAL = {Math. Res. Lett.},
  FJOURNAL = {Mathematical Research Letters},
    VOLUME = {14},
      YEAR = {2007},
    NUMBER = {2},
     PAGES = {249--262},
      ISSN = {1073-2780},
   MRCLASS = {14H45 (14G27 14H37)},
  MRNUMBER = {2318623},
MRREVIEWER = {Aristides I. Kontogeorgis},
   %    DOI = {10.4310/MRL.2007.v14.n2.a8},
       URL = {http://dx.doi.org/10.4310/MRL.2007.v14.n2.a8},
       }
       
       \bib{Hugg05}{article}{
    AUTHOR = {Huggins, B.},
     TITLE = {Fields of Moduli and Fields of Definition of Curves},
   JOURNAL = {Ph.D. Thesis. UCLA},
      YEAR = {2005},
       }
       
\bib{Hur93}{article}{
    AUTHOR = {Hurwitz, A.},
     TITLE = {Uber algebraische Gebilde mit eideuntigen Transformationen in sich},
   JOURNAL = {Math. Ann.},
    VOLUME = {41},
      YEAR = {1893},
     PAGES = {403--442},
       }       
       
      
\bib{Kon09}{article}{
    AUTHOR = {Kontogeorgis, A.},
     TITLE = {Field of moduli versus field of definition for cyclic covers
              of the projective line},
   JOURNAL = {J. Th\'eor. Nombres Bordeaux},
  FJOURNAL = {Journal de Th\'eorie des Nombres de Bordeaux},
    VOLUME = {21},
      YEAR = {2009},
    NUMBER = {3},
     PAGES = {679--692},
      ISSN = {1246-7405},
   MRCLASS = {14H37 (14H30)},
  MRNUMBER = {2605539},
MRREVIEWER = {Meirav Topol},
       URL = {http://jtnb.cedram.org/item?id=JTNB_2009__21_3_679_0},       
}

 
\bib{AKuIKu90}{article}{
    AUTHOR = {Kuribayashi, I.},
    AUTHOR = {Kuribayashi, A.},
     TITLE = {Automorphism groups of compact {R}iemann surfaces of genera
              three and four},
   JOURNAL = {J. Pure Appl. Algebra},
  FJOURNAL = {Journal of Pure and Applied Algebra},
    VOLUME = {65},
      YEAR = {1990},
    NUMBER = {3},
     PAGES = {277--292},
      ISSN = {0022-4049},
     CODEN = {JPAAA2},
   MRCLASS = {30F35 (14H55 20F29)},
  MRNUMBER = {1072285},
MRREVIEWER = {William Harvey},
    %   DOI = {10.1016/0022-4049(90)90107-S},
       URL = {http://dx.doi.org/10.1016/0022-4049(90)90107-S},
       }

       
        
      \bib{LR11}{article}{
    AUTHOR = {Lercier, R.}, 
    AUTHOR= {Ritzenthaler, C.},
     TITLE = {Hyperelliptic curves and their invariants: geometric,
              arithmetic and algorithmic aspects},
   JOURNAL = {J. Algebra},
  FJOURNAL = {Journal of Algebra},
    VOLUME = {372},
      YEAR = {2012},
     PAGES = {595--636},
      ISSN = {0021-8693},
   MRCLASS = {14H45 (13A50 14H10 14Q05)},
  MRNUMBER = {2990029},
MRREVIEWER = {Haohao Wang},
    %   DOI = {10.1016/j.jalgebra.2012.07.054},
       URL = {http://dx.doi.org/10.1016/j.jalgebra.2012.07.054},
}
       
  \bib{LRS13}{article}{
    AUTHOR = {Lercier, R.}, 
    AUTHOR= {Ritzenthaler, C.}, 
    AUTHOR={Sijsling, J.},
     TITLE = {Explicit {G}alois obstruction and descent for hyperelliptic
              curves with tamely cyclic reduced automorphism group},
   JOURNAL = {Math. Comp.},
  FJOURNAL = {Mathematics of Computation},
    VOLUME = {85},
      YEAR = {2016},
    NUMBER = {300},
     PAGES = {2011--2045},
      ISSN = {0025-5718},
     CODEN = {MCMPAF},
   MRCLASS = {14H30 (14H37 14H45)},
  MRNUMBER = {3471117},
 %      DOI = {10.1090/mcom3032},
       URL = {http://dx.doi.org/10.1090/mcom3032},
}


       
      
\bib{MaVa80}{article}{     
   AUTHOR = {Madan, M. L.},
      AUTHOR = {Valentini, R. C.},
     TITLE = {A hauptsatz of {L}. {E}. {D}ickson and {A}rtin-{S}chreier
              extensions},
   JOURNAL = {J. Reine Angew. Math.},
  FJOURNAL = {Journal f\"ur die Reine und Angewandte Mathematik},
    VOLUME = {318},
      YEAR = {1980},
     PAGES = {156--177},
      ISSN = {0075-4102},
     CODEN = {JRMAA8},
   MRCLASS = {12F20},
  MRNUMBER = {579390},
MRREVIEWER = {Sudesh Khanduja},
}

 
\bib{MSSV}{article}{
    AUTHOR = {Magaard, K.},
    AUTHOR = {Shaska, T.},
    AUTHOR = {Shpectorov, S.},
    AUTHOR =  {V{\"o}lklein, H.},
     TITLE = {The locus of curves with prescribed automorphism group},
      NOTE = {Communications in arithmetic fundamental groups (Kyoto,
              1999/2001)},
   JOURNAL = {S\=urikaisekikenky\=usho K\=oky\=uroku},
  FJOURNAL = {S\=urikaisekikenky\=usho K\=oky\=uroku},
    NUMBER = {1267},
      YEAR = {2002},
     PAGES = {112--141},
   MRCLASS = {14H37},
  MRNUMBER = {1954371},
}

\bib{Pau15}{article}{
AUTHOR = {Paulhus, J.},
TITLE = {Branching data for curves up to genus $48$}
YEAR = {2015},
EPRINT = {http://arxiv.org/pdf/1512.07657v1.pdf},
       }
       
\bib{Rey16}{article}{
AUTHOR = {Reyes, C.},
     TITLE = {Pseudoreal Riemann surfaces of small genus},
JOURNAL = {Master Thesis. Universidad de Concepci\'on, Chile},
      YEAR = {2016},
}

\bib{Shi72}{article}{
AUTHOR = {Shimura, G.},
     TITLE = {On the field of rationality for an abelian variety},
   JOURNAL = {Nagoya Math. J.},
  FJOURNAL = {Nagoya Mathematical Journal},
    VOLUME = {45},
      YEAR = {1972},
     PAGES = {167--178},
      ISSN = {0027-7630},
   MRCLASS = {14K15 (14G05 14K10)},
  MRNUMBER = {0306215},
MRREVIEWER = {M. J. Greenberg},
}

\bib{Silv}{book}{
 AUTHOR = {Silverman, J. H.},
     TITLE = {The arithmetic of elliptic curves},
    SERIES = {Graduate Texts in Mathematics},
    VOLUME = {106},
   EDITION = {Second},
 PUBLISHER = {Springer, Dordrecht},
      YEAR = {2009},
     PAGES = {xx+513},
      ISBN = {978-0-387-09493-9},
   MRCLASS = {11-02 (11G05 11G20 14H52 14K15)},
  MRNUMBER = {2514094},
MRREVIEWER = {Vasil{\cprime} {\=I}. Andr{\={\i}}{\u\i}chuk},
%       DOI = {10.1007/978-0-387-09494-6},
       URL = {http://dx.doi.org/10.1007/978-0-387-09494-6},
}

 

\bib{Sin72}{article}{
    AUTHOR = {Singerman, D.},
     TITLE = {Finitely maximal {F}uchsian groups},
   JOURNAL = {J. London Math. Soc. (2)},
  FJOURNAL = {Journal of the London Mathematical Society. Second Series},
    VOLUME = {6},
      YEAR = {1972},
     PAGES = {29--38},
      ISSN = {0024-6107},
   MRCLASS = {30A58},
  MRNUMBER = {0322165},
MRREVIEWER = {A. M. Macbeath},
}

\bib{Sin74}{article}{
    AUTHOR = {Singerman, D.},
     TITLE = {On the structure of non-{E}uclidean crystallographic groups},
   JOURNAL = {Proc. Cambridge Philos. Soc.},
    VOLUME = {76},
      YEAR = {1974},
     PAGES = {233--240},
   MRCLASS = {20H15},
  MRNUMBER = {0348003},
MRREVIEWER = {F. A. Sherk},
}

\bib{Sin74b}{article}{
AUTHOR = {Singerman, D.},
     TITLE = {Symmetries of {R}iemann surfaces with large automorphism
              group},
   JOURNAL = {Math. Ann.},
  FJOURNAL = {Mathematische Annalen},
    VOLUME = {210},
      YEAR = {1974},
     PAGES = {17--32},
      ISSN = {0025-5831},
   MRCLASS = {30A58},
  MRNUMBER = {0361059},
MRREVIEWER = {William Harvey},
}

\bib{Sin80}{article}{
AUTHOR = {Singerman, D.},
     TITLE = {Symmetries and Pseudosymmetries of hyperelliptic surfaces},
   JOURNAL = {Glasgow Math.},
    VOLUME = {21},
      YEAR = {1980},
     PAGES = {39--49},
}


\bib{Weil56}{article}{
AUTHOR = {Weil, A.},
     TITLE = {The field of definition of a variety},
   JOURNAL = {Amer. J. Math.},
  FJOURNAL = {American Journal of Mathematics},
    VOLUME = {78},
      YEAR = {1956},
     PAGES = {509--524},
      ISSN = {0002-9327},
   MRCLASS = {14.0X},
  MRNUMBER = {0082726},
MRREVIEWER = {P. Samuel},
 }
 
 
\bib{Wolfart}{book}{
    AUTHOR = {Jones, G.A.}, 
    AUTHOR = {Wolfart, J.},
     TITLE = {Dessins d'enfants on {R}iemann surfaces},
    SERIES = {Springer Monographs in Mathematics},
 PUBLISHER = {Springer, Cham},
      YEAR = {2016},
     PAGES = {xiv+259},
      ISBN = {978-3-319-24709-0; 978-3-319-24711-3},
   MRCLASS = {14H57 (05C25 11G32 30F10 57M15)},
  MRNUMBER = {3467692},
MRREVIEWER = {Ariyan Javanpeykar},
   %    DOI = {10.1007/978-3-319-24711-3},
       URL = {http://dx.doi.org/10.1007/978-3-319-24711-3},
}
 
 

\newpage

\vspace{2cm}
\section*{Appendix}\label{app}
%\addcontentsline{toc}{chapter}{Appendix: Classification tables}
\vspace{2cm}
 
% \thispagestyle{plain}

\begin{table}[!h]
\centering\renewcommand{\arraystretch}{1.3}\setlength{\tabcolsep}{2pt}

    \begin{tabular}{| c | c | c | c | p{3.5cm} |}
    \hline
    \multicolumn{5}{ |c| }{Genus $2$} \\
    \hline
    $\rm Aut^{+}(X)$ & Fuchsian signature & $\rm Aut^{\pm}(X)$ & NEC signature & Generating Vector \\ \hline
    $C_2$ & $(2^6)$ & $C_4$ & $(1;-;[2^3])$ & $(a;a^2,a^2,a^2)$ \\ \hline
    \end{tabular}
    \vspace{0.5cm}
    \caption{Automorphism groups of pseudoreal Riemann surfaces of genus 2}\label{GEN2}
\end{table}

\medskip

\begin{table}[!h]
\centering\renewcommand{\arraystretch}{1.3}\setlength{\tabcolsep}{2pt}


    \begin{tabular}{| c | c | c | c | p{3.5cm} |}
    \hline
    \multicolumn{5}{ |c| }{Genus $3$} \\
    \hline
    $\rm Aut^{+}(X)$ & Fuchsian signature & $\rm Aut^{\pm}(X)$ & NEC signature & Generating Vector \\ \hline
    $C_2$ & $(1;[2^4])$ & $C_4$ & $(2;-;[2^2])$ & $(a,a;a^2,a^2)$ \\ \hline
    $C_2\times C_2$ & $(0;[2^6])$ & $C_4\times C_2$ & $(1;-;[2^3])$ & $(a;b,b,a^2)$ \\ \hline
    \end{tabular}
    \vspace{0.5cm}
    \caption{Automorphism groups of pseudoreal Riemann surfaces of genus 3}
\label{GEN3}
\end{table}

\medskip

\begin{table}[!h]
\centering\renewcommand{\arraystretch}{1.3}\setlength{\tabcolsep}{2pt}
 
    \begin{tabular}{| c | c | c | c | p{3.5cm} |}
    \hline
    \multicolumn{5}{ |c| }{Genus $4$} \\
    \hline
    $\rm Aut^{+}(X)$ & Fuchsian signature & $\rm Aut^{\pm}(X)$ & NEC signature & Generating Vector \\ \hline
    $C_2$ & $(0;[2^{10}])$ & $C_4$ & $(1;-;[2^5])$ & $(a;a^2,a^2,a^2,a^2,a^2)$ \\ \hline
    $C_2$ & $(2;[2^2])$ & $C_4$ & $(3;-;[2])$ & $(a,a,a;a^2)$ \\ \hline
    $C_4$ & $(0;[2^4,4^2])$ & $C_8$ & $(1;-;[2^2,4])$ & $(a^3;a^4,a^4,a^2)$ \\ \hline
    $D_5$ & $(0;[2^2,5^2])$ & $F20$ & $(1;-;[2,5])$ & $(b;b^2a,a^4)$ \\ \hline
    \end{tabular}
     \vspace{0.5cm}
\caption{Automorphism groups of pseudoreal Riemann surfaces of genus 4}
\label{GEN4}
\end{table}
\newpage

\vspace{5cm} 
\ \ \
 

\begin{table}[!h]
\centering\renewcommand{\arraystretch}{1.3}\setlength{\tabcolsep}{2pt}

 

    \begin{tabular}{| c | c | c | c | p{3.5cm} |}
    \hline
    \multicolumn{5}{ |c| }{Genus $5$} \\
    \hline
    $\rm Aut^{+}(X)$ & Fuchsian signature & $\rm Aut^{\pm}(X)$ & NEC signature & Generating Vector \\ \hline
    $C_2$ & $(3;[-])$ & $C_4$ & $(4;-;[-])$ & $(a,a,a,a;[-])$ \\ \hline
    $C_2$ & $(1;[2^8])$ & $C_4$ & $(2;-;[2^4])$ & $(a,a;a^2,a^2,a^2,a^2)$ \\ \hline
    $C_4$ & $(1;[2^4])$ & $C_8$ & $(2;-;[2^2])$ & $(a,a^3;a^4,a^4)$ \\ \hline
    $C_4$ & $(1;[2^4])$ & $Q_8$ & $(2;-;[2^2])$ & $(j,k;-1,-1)$ \\ \hline    
    $C_4$ & $(0;[2^2,4^4])$ & $Q_8$ & $(1;-;[2,4^2])$ & $(j;-1,i,-i)$ \\ \hline
    $C_2\times C_2$ & $(0;[2^8])$ & $C_4\times C_2$ & $(1;-;[2^4])$ & $(a;b,b,b,a^2b)$ \\ \hline
    $C_2\times C_2$ & $(1;[2^4])$ & $C_4\times C_2$ & $(2;-;[2^2])$ & $(a,a;b,b)$ \\ \hline
    $C_6$ & $(1;[3^2])$ & $C_{12}$ & $(2;-;[3])$ & $(a,a;a^8)$ \\ \hline
    $D_4$ & $(0;[2^6])$ & $QD_8$ & $(1;-;[2^3])$ & $(xa;a^4,a^4,a^4)$ \\ \hline
    $C_2\times C_2\times C_2$ & $(0;[2^6])$ & $C_4\times C_2\times C_2$ & $(1;-;[2^3])$ & $(a;b,c,a^2bc)$ \\ \hline
     & $-$ & $C_4\times C_2\rtimes_{\phi} C_2$ & $(1;-;[2^3])$ & $(a;b,c,a^2bc)$ \\ \hline
    \end{tabular}
    \vspace{0.5cm}
    \caption{Automorphism groups of pseudoreal Riemann surfaces of genus 5}\label{table5}
    \end{table}
     
%\vspace{2cm} 

\begin{table}[!h]
\centering\renewcommand{\arraystretch}{1.3}\setlength{\tabcolsep}{2pt}
 

    \begin{tabular}{| c | c | c | c | p{4.5cm} |}
    \hline
    \multicolumn{5}{ |c| }{Genus $6$} \\
    \hline
    $\rm Aut^{+}(X)$ & Fuchsian signature & $\rm Aut^{\pm}(X)$ & NEC signature & Generating Vector \\ \hline
    $C_2$ & $(0;[2^{14}])$ & $C_4$ & $(1;-;[2^{7}])$ & $(a;a^2,a^2,a^2,a^2,a^2,a^2,a^2)$ \\ \hline
    $C_2$ & $(2;[2^6])$ & $C_4$ & $(3;-;[2^{3}])$ & $(a,a,a;a^2,a^2,a^2)$ \\ \hline
    $C_4$ & $(0;[2^6,4^2])$ & $C_8$ & $(1;-;[2^3,4])$ & $(a;a^4,a^4,a^4,a^2)$ \\ \hline
    $C_4$ & $(0;[4^6])$ & $C_8$ & $(1;-;[4^3])$ & $(a;a^2,a^2,a^2)$ \\ \hline
    $C_6$ & $(0;[2^4,6^2])$ & $C_{12}$ & $(1;-;[2^2,6])$ & $(a;a^6,a^6,a^{10})$ \\ \hline
    $C_6$ & $(0;[2^2,3^4])$ & $C_{12}$ & $(1;-;[2,3^2])$ & $(a^5;a^6,a^4,a^4)$ \\ \hline
     & $-$ & ${\rm Dic}_{12}$ & $(1;-;[2,3^2])$ & $(x;a^3,a^2,a^4)$ \\ \hline
     $D_5$ & $(0;[2^6])$ & $F20$ & $(1;-;[2^3])$ & $(b;ab^2,ab^2,b^2)$ \\ \hline
    \end{tabular}
      \vspace{0.5cm}
    \caption{Automorphism groups of pseudoreal Riemann surfaces of genus 6}\label{table6}
\end{table}

\newpage

\begin{table}[!h]
\centering\renewcommand{\arraystretch}{1.3}\setlength{\tabcolsep}{2pt}

    \begin{tabular}{| c | c | c | c | p{4.3cm} |}
    \hline
    \multicolumn{5}{ |c| }{Genus $7$} \\
    \hline
    $\rm Aut^{+}(X)$ & Fuchsian signature & $\rm Aut^{\pm}(X)$ & NEC signature & Generating Vector \\ \hline
    $C_2$ & $(1;[2^{12}])$ & $C_4$ & $(2;-;[2^6])$ & $(a,a;a^2,a^2,a^2,a^2,a^2,a^2)$ \\ \hline
    $C_2$ & $(3;[2^{4}])$ & $C_4$ & $(4;-;[2^2])$ & $(a,a,a,a;a^2,a^2)$ \\ \hline
     $C_4$ & $(0;[2^4,4^4])$ & $Q_8$ & $(1;-;[2^2,4^2])$ & $(j;-1,-1,i,i)$ \\ \hline
     $C_4$ & $(1;[2^6])$ & $C_8$ & $(2;-;[2^3])$ & $(a,a;a^4,a^4,a^4)$ \\ \hline
     $C_4$ & $(1;[4^4])$ & $C_8$ & $(2;-;[4^2])$ & $(a,a;a^2,a^2)$ \\ \hline
     & $-$ & $Q_8$ & $(2;-;[4^2])$ & $(j,j;i,-i)$ \\ \hline
     $C_4$ & $(2;[2^2])$ & $Q_8$ & $(3;-;[2])$ & $(j,j,j;-1)$ \\ \hline
     $C_2\times C_2$ & $(0;[2^{10}])$ & $C_4\times C_2$ & $(1;-;[2^5])$ & $(a;b,b,a^2,a^2,a^2)$ \\ \hline
     $C_2\times C_2$ & $(1;[2^6])$ & $C_4\times C_2$ & $(2;-;[2^3])$ & $(a,a;a^2,a^2b,b)$ \\ \hline
     $C_2\times C_2$ & $(2;[2^2])$ & $C_4\times C_2$& $(3;-;[2])$ & $(a,a,ab;a^2)$ \\ \hline
     $C_6$ & $(1;[2^4])$ & $C_{12}$& $(2;-;[2^2])$ & $(a^3,a^3;a^6,a^6)$ \\ \hline
     & $(1;[2^4])$ & ${\rm Dic}_{12}$& $(2;-;[2^2])$ & $(ax,x;a^3,a^3)$ \\ \hline
     $C_4\times C_2$ & $(0;[2^4,4^2])$ & $C_8\times C_2$ & $(1;-;[2^2,4])$ & $(a;b,b,a^6)$ \\ \hline
      & $-$ & $C_8\rtimes_{\phi} C_2$ & $(1;-;[2^2,4])$ & $(ax;a^4x,a^4x,a^2)$ \\ \hline
     $D_4$ & $(0;[2^4,4^2])$ & $QD_8$& $(1;-;[2^2,4])$ & $(a;x,x,a^6)$ \\ \hline
     $D_6$ & $(0;[2^6])$ & $C_3\times S_3$& $(1;-;[2^3])$ & $-$ \\ \hline
    \end{tabular}
       \vspace{0.5cm}
\caption{Automorphism groups of pseudoreal Riemann surfaces of genus 7}\label{table7}
\end{table}

\medskip

\begin{table}[!h]
\centering\renewcommand{\arraystretch}{1.3}\setlength{\tabcolsep}{2pt}
 

    \begin{tabular}{| c | c | c | c | p{5.7cm} |}
    \hline
    \multicolumn{5}{ |c| }{Genus $8$} \\
    \hline
    $\rm Aut^{+}(X)$ & Fuchsian signature & $\rm Aut^{\pm}(X)$ & NEC signature & Generating Vector \\ \hline
    $C_2$ & $(0;[2^{18}])$ & $C_4$ & $(1;-;[2^9])$ & $(a;a^2,a^2,a^2,a^2,a^2,a^2,a^2,a^2,a^2)$ \\ \hline
    $C_2$ & $(2;[2^{10}])$ & $C_4$ & $(3;-;[2^5])$ & $(a,a,a;a^2,a^2,a^2,a^2,a^2)$ \\ \hline
    $C_2$ & $(4;[2^2])$ & $C_4$ & $(5;-;[2])$ & $(a,a,a,a,a;a^2)$ \\ \hline
     $C_4$ & $(0;[2^8,4^2])$ & $C_8$ & $(1;-;[2^4,4])$ & $(a^3;a^4,a^4,a^4,a^4,a^2)$ \\ \hline
     $C_4$ & $(0;[2^2,4^6])$ & $C_8$ & $(1;-;[2,4^3])$ & $(a^3;a^4,a^2,a^2,a^2)$ \\ \hline
     $C_4$ & $(2;[4^2])$ & $C_8$ & $(3;-;[4])$ & $(a,a,a;a^2)$ \\ \hline
     $C_6$ & $(0;[3^4,6^2])$ & $C_{12}$ & $(1;-;[3^2,6])$ & $(a;a^4,a^4,a^2)$ \\ \hline
     & $-$ & ${\rm Dic}_{12}$ & $(1;-;[3^2,6])$ & $(x;a^2,a^2,a^5)$ \\ \hline
     $C_6$ & $(0;[2^6,3^2])$ & $C_{12}$ & $(1;-;[2^3,3])$ & $(a;a^6,a^6,a^6,a^4)$ \\ \hline
     $C_6$ & $(0;[2^2,6^4])$ & $C_{12}$ & $(1;-;[2,6^2])$ & $(a;a^6,a^2,a^2)$ \\ \hline
     & $-$ & ${\rm Dic}_{12}$ & $(1;-;[2,6^2])$ & $(x;a^3,a,a^5)$ \\ \hline
     $C_8$ & $(0;[2^4,8^2])$ & $C_{16}$ & $(1;-;[2^2,8])$ & $(a;a^8,a^8,a^{14})$ \\ \hline
    \end{tabular}
    \vspace{0.5cm}
\caption{Automorphism groups of pseudoreal Riemann surfaces of genus 8}\label{table8}
\end{table}

\newpage
\vspace{5cm}
\ \\
\begin{table}[!h]
\centering\renewcommand{\arraystretch}{1.3}\setlength{\tabcolsep}{2pt}

    \begin{tabular}{| c | c | c | c |}
    \hline
    \multicolumn{4}{ |c| }{Genus $9$} \\
    \hline
    $\rm Aut^{+}(X)$ & Fuchsian signature & $\rm Aut^{\pm}(X)$ & NEC signature \\ \hline
    $C_2$ & $(1;[2^{16}])$ & $C_4$ & $(2;-;[2^8])$ \\ \hline
    $C_2$ & $(3;[2^{8}])$ & $C_4$ & $(4;-;[2^4])$ \\ \hline
    $C_2$ & $(5;[-])$ & $C_4$ & $(6;-;[-])$\\ \hline
     $C_4$ & $(0;[4^8])$ & $Q_8$ & $(1;-;[4^4])$ \\ \hline
     $C_4$ & $(0;[2^6,4^4])$ & $Q_8$ & $(1;-;[2^3,4^2])$ \\ \hline
     $C_4$ & $(1;[2^8])$ & $C_8$ & $(2;-;[2^4])$ \\ \hline
      & $-$ & $Q_8$ & $(2;-;[2^4])$ \\ \hline
     $C_4$ & $(1;[2^2,4^4])$ & $C_8$ & $(2;-;[2,4^2])$ \\ \hline
      & $-$ & $Q_8$ & $(2;-;[2,4^2])$ \\ \hline    
      $C_4$ & $(3;[-])$ & $C_8$ & $(4;-;[-])$ \\ \hline 
      & $-$ & $Q_8$ & $(4;-;[-])$ \\ \hline 
      $C_2\times C_2$ & $(0;[2^{12}])$ & $C_4\times C_2$ & $(1;-;[2^6])$ \\ \hline       
      $C_2\times C_2$ & $(1;[2^8])$ & $C_4\times C_2$ & $(2;-;[2^4])$ \\ \hline 
      $C_2\times C_2$ & $(2;[2^4])$ & $C_4\times C_2$ & $(3;-;[2^2])$ \\ \hline 
      $C_2\times C_2$ & $(3;[-])$ & $C_4\times C_2$ & $(4;-;[-])$ \\ \hline 
      $C_6$ & $(1;[2^2,6^2])$ & $C_{12}$ & $(1;-;[2^6])$ \\ \hline 
      $C_6$ & $(1;[3^4])$ & ${\rm Dic}_{12}$ & $(2;-;[3^2])$ \\ \hline 
       & $-$ & $C_{12}$ & $(2;-;[3^2])$ \\ \hline 
      $D_4$ & $(0;[2^8])$ & $QD_8$ & $(1;-;[2^4])$ \\ \hline
      $C_2^3$ & $(0;[2^8])$ & ${\rm ID}(16,3)$ & $(1;-;[2^4])$ \\ \hline
       & $-$ & $C_4\times C_2^2$ & $(1;-;[2^4])$ \\ \hline
      $C_4\times C_2$ & $(0;[2^2,4^4])$ & $C_4\times C_4$ & $(1;-;[2,4^2])$ \\ \hline
       & $-$ & ${\rm ID}(16,4)$ & $(1;-;[2,4^2])$ \\ \hline
        & $-$ & $C_2\times Q_8$ & $(1;-;[2,4^2])$ \\ \hline
        $Q_8$ & $(0;[2^2,4^4])$ & $Q_{16}$ & $(1;-;[2,4^2])$ \\ \hline
    $C_8$ & $(1;[2^4])$ & $C_{16}$ & $(2;-;[2^2])$ \\ \hline    
    & $-$ & $Q_{16}$ & $(2;-;[2^2])$ \\ \hline  
        \end{tabular}
        \vspace{0.5cm}
\caption{Automorphism groups of pseudoreal Riemann surfaces of genus 9}\label{table9}
\end{table}

\newpage
\vspace{5cm}
\ \\

\begin{table}[!h]
\centering\renewcommand{\arraystretch}{1.3}\setlength{\tabcolsep}{2pt}
\label{table9pt2}
\medskip

    \begin{tabular}{| c | c | c | c |}
    \hline
    \multicolumn{4}{ |c| }{Genus $9$ (continuation)} \\
    \hline
    $\rm Aut^{+}(X)$ & Fuchsian signature & $\rm Aut^{\pm}(X)$ & NEC signature \\ \hline  
$C_4\times C_2$ & $(1;[2^4])$ & $C_4\times C_4$ & $(2;-;[2^2])$ \\ \hline       
 & $-$ & ${\rm ID}(16,4)$ & $(2;-;[2^2])$ \\ \hline 
  & $-$ & $C_8\times C_2$ & $(2;-;[2^2])$ \\ \hline 
   & $-$ & ${\rm ID}(16,6)$ & $(2;-;[2^2])$ \\ \hline 
    & $-$ & $C_2\times Q_8$ & $(2;-;[2^2])$ \\ \hline    
 $D_4$ & $(1;[2^4])$ & $QD_8$ & $(2;-;[2^2])$ \\ \hline       
 $C_2^3$ & $(1;[2^4])$ & ${\rm ID}(16,3)$ & $(2;-;[2^2])$ \\ \hline       
 & $-$ & $C_4\times C_2^2$ & $(2;-;[2^2])$ \\ \hline 
$C_4\times C_2$ & $(2;[-])$ & ${\rm ID}(16,6)$ & $(3;-;[-])$ \\ \hline     
$D_4$ & $(2;[-])$ & $QD_8$ & $(3;-;[-])$ \\ \hline          
$D_5$ & $(1;[5^2])$ & $F20$ & $(2;-;[5])$ \\ \hline        
$C_{10}$ & $(1;[5^2])$ & $C_{20}$ & $(2;-;[5])$ \\ \hline 
$D_6$ & $(0;[2^4,3^2])$ & $C_4\times S_3$ & $(2;-;[2^2,3])$ \\ \hline 
$C_6\times C_2$ & $(0;[2^4,3^2])$ & $C_{12}\times C_2$ & $(2;-;[2^2,3])$ \\  \hline 
$C_{12}$ & $(1;[3^2])$ & $C_{24}$ & $(2;-;[3])$ \\  \hline 
& $-$ & $C_3\times Q_8$ & $(2;-;[3])$ \\ \hline
$C_6\times C_2$ & $(1;[3^2])$ & $C_{12}\times C_2$ & $(2;-;[3])$ \\  \hline 
$D_8$ & $(0;[2^6])$ & ${\rm ID}(32,19)$ & $(1;-;[2^3])$ \\  \hline 
$C_2\times D_4$ & $(0;[2^6])$ & ${\rm ID}(32,6)$ & $(1;-;[2^3])$ \\  \hline 
& $-$ & ${\rm ID}(32,7)$ & $(1;-;[2^3])$ \\ \hline 
& $-$ & ${\rm ID}(32,9)$ & $(1;-;[2^3])$ \\ \hline 
& $-$ & $C_2\times QD_8$ & $(1;-;[2^3])$ \\ \hline 
${\rm ID}(16,13)$ & $(0;[2^6])$ & ${\rm ID}(32,11)$ & $(1;-;[2^3])$ \\  \hline
 & $-$ & ${\rm ID}(32,38)$ & $(1;-;[2^3])$ \\  \hline
$C_2^4$ & $(0;[2^6])$ & ${\rm ID}(32,22)$ & $(1;-;[2^3])$ \\  \hline
${\rm ID}(16,6)$ & $(1;[2^2])$ & ${\rm ID}(32,15)$ & $(2;-;[2])$ \\  \hline
$D_{10}$ & $(0;[2^2,10^2])$ & ${\rm ID}(40,12)$ & $(1;-;[2,10])$ \\  \hline

  \end{tabular}
\end{table}   


\newpage

\begin{table}[!h]
\centering\renewcommand{\arraystretch}{1.3}\setlength{\tabcolsep}{2pt}
    \begin{tabular}{| c | c | c | c |}
    \hline
    \multicolumn{4}{ |c| }{Genus $10$} \\
    \hline
    $\rm Aut^{+}(X)$ & Fuchsian signature & $\rm Aut^{\pm}(X)$ & NEC signature \\ \hline
    $C_2$ & $(0;[2^{22}])$ & $C_4$ & $(1;-;[2^{11}])$ \\ \hline
    $C_2$ & $(2;[2^{14}])$ & $C_4$ & $(3;-;[2^{7}])$ \\ \hline
   $C_2$ & $(4;[2^{6}])$ & $C_4$ & $(5;-;[2^{3}])$ \\ \hline 
$C_4$ & $(0;[2^{10},4^2])$ & $C_8$ & $(1;-;[2^{5},4])$ \\ \hline    
$C_4$ & $(0;[2^{4},4^6])$ & $C_8$ & $(1;-;[2^{2},4^3])$ \\ \hline    
$C_4$ & $(2;[2^{2},4^2])$ & $C_8$ & $(3;-;[2,4])$ \\ \hline    
$C_6$ & $(0;[2^4,3^2,6^2])$ & ${\rm Dic}_{12}$ & $(1;-;[2^2,3,6])$ \\ \hline     
 & $-$ & $C_{12}$ & $(1;-;[2^2,3,6])$ \\ \hline   
$C_6$ & $(0;[2^2,3^6])$ & ${\rm Dic}_{12}$ & $(1;-;[2,3^3])$ \\ \hline     
 & $-$ & $C_{12}$ & $(1;-;[2,3^3])$ \\ \hline      
$C_6$ & $(0;[6^6])$ & ${\rm Dic}_{12}$ & $(1;-;[6^3])$ \\ \hline     
 & $-$ & $C_{12}$ & $(1;-;[6^3])$ \\ \hline     
$C_6$ & $(2;[2^2])$ & ${\rm Dic}_{12}$ & $(3;-;[2])$ \\ \hline     
 & $-$ & $C_{12}$ & $(3;-;[2])$ \\ \hline     
$C_8$ & $(0;[2^2,4^2,8^2])$ & $C_{16}$ & $(1;-;[2,4,8])$ \\ \hline  
$C_{10}$ & $(0;[2^4,10^2])$ & $C_{20}$ & $(1;-;[2^2,10])$ \\ \hline 
${\rm ID}(18,4)$ & $(0;[2^6])$ & ${\rm ID}(36,9)$ & $(1;-;[2^3])$ \\ \hline  
${\rm ID}(36,9)$ & $(0;[2^2,4^2])$ & ${\rm ID}(72,39)$ & $(1;-;[2,4])$ \\ \hline   
    \end{tabular}
    \vspace{0.5cm}
\caption{Automorphism groups of pseudoreal Riemann surfaces of genus 10}\label{table10}
\end{table}
 
 
  
 
 
  

 

 
\end{biblist}
\end{bibdiv}
\end{document}